\documentclass{article}
\usepackage{amsthm,amsmath,amsfonts,amssymb}
\usepackage{enumitem}
\usepackage{mathrsfs}
\usepackage{color}
\usepackage{hyperref}
\usepackage[numbers]{natbib}

\title{Efficient Bayesian Inference in Strictly Semi-parametric Linear Inverse Problems}

\author{Adel Magra, Aad Van der Vaart}

\newcommand{\R}{\mathcal{R}}

\newcommand{\ip}[2]{\left \langle #1, #2 \right \rangle} 
\newcommand{\eps}{\varepsilon}
\newcommand{\q}{\theta}
\newcommand{\RR}{\mathbb{R}}
\newcommand{\CC}{\mathbb{C}}
\newcommand{\gaatp}{\overset{P}{\to}}
\newcommand{\g}{\gamma}
\newcommand{\D}{\mathcal{D}}
\newcommand{\HHH}{\mathbb{H}}

\newcommand{\p}{\varphi}
\newcommand{\DD}{\mathbb{D}}
\newcommand{\KK}{\mathcal{K}}

\theoremstyle{plain}
\newtheorem{thm}{Theorem}[section]
\newtheorem{lemma}[thm]{Lemma}
\newtheorem{cor}[thm]{Corollary}
\newtheorem{prop}[thm]{Proposition}

\newtheorem{asspt}[thm]{Assumption}
\newtheorem{claim}[thm]{Claim}
\newtheorem{model}{Model}

\newtheorem{condition}[thm]{Condition}
\theoremstyle{definition}

\begin{document}
	
\maketitle

\begin{abstract}
	We consider the efficient inference of finite dimensional parameters arising in the context of inverse problems. Our setup is the observation of a transformation of an unknown infinite dimensional signal $f$ corrupted by statistical noise, with the transformation $K_\theta$ being linear but unknown up to a scalar $\theta$. We adopt a Bayesian approach and put a prior on the pair $(\theta,f)$ and prove a Bernstein-von Mises theorem for the marginal posterior of $\theta$ under regularity conditions on the operators $K_\theta$ and on the prior. We apply our results to the recovery of location parameters in semi-blind deconvolution problems and to the recovery of attenuation constants in X-ray tomography. 
\end{abstract}

\section{Introduction}

Let $\Theta\subset \mathbb{R}$ be a compact set, and let $(H_1,\ip{\cdot}{\cdot}_{H_1}), (H_2, \ip{\cdot}{\cdot}_{H_2})$ be (infinite dimensional) separable Hilbert spaces. Consider a family of continuous linear operators $\{K_\theta: H_1\to H_2,\ \theta\in \Theta\}$. Suppose we have access to the following signal in white noise data,
\begin{align}\label{eq: obs}
	X^{(n)} = K_\theta f + \frac{1}{\sqrt{n}}\dot W,
\end{align}
where $n\in\mathbb{N}$ is the signal to noise ratio and with $\dot W$ the iso-Gaussian process for $H_2$ (which will typically be a space of square integrable functions). We take a Bayesian approach and endow the pair $(\theta, f)$ with a product prior $\Pi = \pi_\theta\times \pi_f$, with $\pi_\theta$ having positive density with respect to (w.r.t) the Lebesgue measure on $\Theta$. We assume that there exists a true pair $(\theta_0, f_0)$ that generated the data \eqref{eq: obs}. We are interested in uncertainty quantification results for $\theta_0$ in the asymptotic regime, i.e. when $n\to\infty$. 

Our main result is a Bernstein-von Mises (BvM) theorem for the marginal posterior of $\theta$. This is obtained first through ensuring that the model provides sufficient identifiability for $\theta$, and further provided that $K_\theta f$ is sufficiently regular in both $\theta$ and $f$.

Bernstein-von Mises theorems in nonparametric Bayesian inverse problems have received considerable attention in recent years. Typically, one observes a noisy version of $Kf$, with $K$ a known operator and $f$ an unknown infinite dimensional parameter, and what is sought is to derive (semi-parametric) BvM theorems for finite dimensional aspects of $f$, i.e. linear functionals of $f$. This has been studied for the case of linear $K$ (e.g. in \cite{Knapiketal2011, GiordanoKekkonen}) as well as nonlinear $K$ (e.g. in \cite{Nickl2018, monard2019, MonardNicklPternain2021}), with the latter presenting considerably more challenges. 

In our context, the uncertainty is both on the operator $K$ and the function $f$. Although the map $f\mapsto K_\theta f$ is linear, the map $(\theta, f)\mapsto K_\theta f$ typically isn't. Our setup does therefore technically fall in the category of nonparametric Bayesian nonlinear inverse problems. The approach typically taken in those  models to obtain BvM results is to first prove contraction of the full posterior at the truth $(\theta_0,f_0$) at a polynomial rate (w.r.t $n$), and then to use this localisation to control the difference of the log-likelihood ratio with its local asymptotic normal counterpart (see Chapter 4 in\cite{Nickl23} for a detailed overview). This approach has  already been undertaken in our ``strict" model (in the sense that the parametric part of the unknown $\theta$ is separated from the nonparametric part $f$) in \cite{magra_VdM_VdV} but with a specific nonlinear $f\mapsto K_\theta f$ stemming from the solution map of a parabolic PDE 

The fact that the operators $K_\theta$ are linear in our setup will allow for more leeway with regards to contraction. We will show that contraction of the full posterior at $(\theta_0, K_{\theta_0}f_0)$ (and potentially at $\dot K_{\theta_0}f_0$, the derivative of $K_\theta f$ w.r.t $\theta$) is enough of a requirement in our case. For compact $K_\theta$ we have that $K_\theta f$ is typically much smoother than $f$, and so obtaining polynomial rates of contraction can be considerably easier. This is particularly useful in severely ill-posed inverse problems where BvM results based on contraction at $f_0$ are hopeless, whereas contraction at $K_{\theta_0} f_0$ may be achieved at an almost parametric level. Interestingly, increased ill-posedness is beneficial in the following context of inverse problems. This phenomenon has also been observed in a similar problem in \cite{magra_VdV_HvZ} where one has access to  \eqref{eq: obs} in addition to a noisy observation of $f$.

We present our main BvM results in Section 2. We specify them to a class of re-scaled Gaussian priors in Section 3 and apply them to semi-blind deconvolution problems in Section 4 and to attenuated X-ray transforms in Section 5. Section \ref{sec: proofs} is dedicated to proofs of results not found in other sections.

\section{Main Results}

\subsection{Model Identifiability and Information}

Let $(\Theta,H)\subseteq (\RR,H_1)$ be the parameter set for $(\theta,f)$, where $\Theta$ is a compact subset of $\RR$ and $H$ is a linear subspace of $H_1$ equipped with a norm $\|\cdot\|_H$. For any pair $(\theta,f)\in\Theta\times H$, we let $P_{\theta,f}^n$ be the distribution of the process in \eqref{eq: obs}. Letting $P_0^n$ be the distribution of $\sqrt{n}^{-1}\dot{W}$, we obtain the following expression for the log-likelihood,
\begin{align}\label{eq: likelihood}
	\log{\dfrac{dP_{\theta,f}^n}{dP_0^n}}(X^{(n)}) = n\ip{K_\theta f}{X^{(n)}}_{H_2} - \dfrac{n}{2}\|K_\theta f\|_{H_2}^2.
\end{align}
We assume the existence of a derivative for the map $\theta\mapsto K_\theta f$ on $\Theta$, defined as follows,
\begin{align*}
	\dfrac{K_{\theta+t}f - K_\theta f}{t} \to \dot K_\theta f \text{ as }t\to 0,
\end{align*}
where the limit is taken in the $H_2$-norm. To study our model, we consider one-dimensional sub-models of the type $t\mapsto (\theta + t, f + gt)$ for a given direction $g\in H$. The linearity of $K_\theta$ implies that the following limit holds in the $H_2$-norm as $t\to0$,
\begin{align*}
	\dfrac{K_{\theta+t}(f+gt) - K_\theta f}{t} = \dfrac{K_{\theta+t}f - K_{\theta}f}{t} + {K_{\theta+t}g} \to \dot K_\theta f + K_\theta g.
\end{align*}
The local asymptotic normality (LAN) of this sub-model can subsequently be derived by looking at the limit of the log-likelihood ratios as $n\to\infty$,
 \[\log{\dfrac{dP^n_{\theta+t/\sqrt{n}, f+gt/\sqrt{n}}}{dP^n_{\theta, f}}}\to t\ip{\dot K_\theta f + K_\theta g}{\dot W}_{H_2} -\dfrac{1}{2}t^2\|\dot K_\theta f + K_\theta g\|_{H_2}^2. \] 
 The Fisher information for estimating $\theta$ in this sub-model (at $t=0$) is $\|\dot K_\theta f + K_\theta g\|^2.$ The efficient Fisher information is the infimum of this quantity over all possible directions $g\in H$. This is easily seen as the projection of $-\dot K_\theta f$ onto the closure of $K_\theta H$. Consequently, if a function $-\gamma_{\theta, f}\in H$ satisfies,
 \begin{align}\label{eq: lfd}
	\ip{\dot K_\theta f - K_\theta \gamma_{\theta,f}}{K_\theta h}_{H_2}=0,\ \forall h\in H,
 \end{align} 
 then it is a minimiser of the Fisher information. We call $\gamma_{\theta,f}$ a \textit{least favourable direction} (LFD) because it corresponds to a sub-model yielding the smallest Fisher information, namely the efficient information which is then equal to,
 \begin{align}\label{eq: info}
 	\tilde{i}_{\theta,f} = \|\dot K_\theta f - K_\theta \gamma_{\theta, f}\|_{H_2}^2 = \|\dot K_\theta f\|_{H_2}^2 - \|K_\theta \gamma_{\theta,f}\|_{H_2}^2.
 \end{align}
 In the special case that $\gamma_{\theta,f}=0$, estimation of $\theta$ in all sub-models is equivalent and the efficient information is $\|\dot K_\theta f\|^2$, which is the information in the model where $f$ is known. This generally does not happen unless there is some sort of nice relationship between $K_\theta$ and the space $H$. The last display also makes apparent that $\tilde{i}_{\theta,f} < \|\dot K_\theta f\|^2$ in general. Note also that a least favourable direction need not exist, or that there might be multiple least favourable directions depending on the model. In the latter case they will all however yield the same efficient information which is given by \eqref{eq: info}.
 

To start thinking of a BvM result for the marginal posterior of $\theta$, it is essential that $\tilde{i}_{\theta,f}>0$. We will assume throughout that this is the case. In concrete applications, this assumption can be verified based on the structure of the model considered and through assumptions on the true pair of parameters $(\theta, f)$ .



\subsection{A first BvM}

Assume that a true pair $(\theta_0,f_0)\in \Theta\times H$ generates the data in model \eqref{eq: obs}. We consider a product prior $\Pi:=\pi_\theta\times\pi_f$, where $\pi_\theta$ has a positive density w.r.t the Lebesgue measure, and with $\pi_f$ a distribution supported on $H$. Saying that the \textit{Bernstein-von Mises theorem holds at $(\theta_0, f_0)$} means that in $P_{\theta_0, f_0}^n$-probability, as $n\to\infty$, 
\begin{align}
	\Big\|\Pi(\theta \in \cdot \mid X^{(n)}) - 
	N\Big(\theta_0 + \frac{1}{\sqrt{n}} \Delta_{ \theta_0, f_0}^n, 
	\frac{1}{n}\tilde i^{-1}_{\theta_0, f_0}\Big)\Big\|_{\text{TV}}\to 0,
	\label{EqBvMTheorem}
\end{align}
where  $\Delta_{ \theta_0, f_0 }^n$ are measurable
transformations of $X^{(n)}$ such that $\Delta_{ \theta_0, f_0 }^n\sim N(0, \tilde i^{-1}_{\theta_0, f_0})$. 
In all our results the latter variables satisfy
\[\Delta_{\theta,f}^n
=\frac{1}{\tilde i_{\theta,f}}\ip{\dot K_{\theta}f-K_{\theta}\gamma_{\theta,f}}{\dot W}.\]
In other words, the variables $\Delta_{ \theta_0, f_0 }^n$ are the efficient score functions at $(\theta_0,f_0)$ divided by
the efficient Fisher information. The BvM presented in \eqref{EqBvMTheorem} uses the
total variation norm, given by $\|P-Q\|_{\text{TV}}:=2\sup_{F\in\mathcal{F}} |P(F)-Q(F)|$.





We prove a BvM under general conditions. We assume that there exists a normed space $S\subset H_2$ with a stronger norm $\|\cdot\|_S$ than $H_2$. We suppose that $K_\theta H\subset S$. If $\|\cdot\|_{H}\approx\|\cdot\|_{H_1}$, then this means that $K_\theta$ is ``smoothing", as is the case in ill-posed problems like the examples of Section 4. However this is not a necessity as the X-ray transform example of Section 5 illustrates. A key requirement is that the unit ball in $S$ must have a finite entropy integral in $H_2$:

\begin{align}\label{eq: finiteEntropy}
	\int_0^1 \sqrt{\log{N(\eps, \{h:\|h\|_S\leq 1\}, \|\cdot\|_{H_2})}}d\eps < \infty,
\end{align}
where for a metric space $(A,d)$, $N(\eps, A, d)$ is the minimum number of $d$-balls of radius $\eps$ needed to cover $A$.

We further make the following regularity assumptions on the map $\theta\mapsto K_\theta$.

\begin{asspt}\label{asspt1}
There exist constants $D_1, D_2, D_3$ and $D_4$ such that $\forall (\theta,\theta_0)\in\Theta$ and $f\in H$,
\begin{enumerate}[label=(\roman*),leftmargin=*, itemsep=0.5ex, before={\everymath{\displaystyle}}]
	\item $\|K_\theta f\|_{H_2} \leq D_1\|f\|_{H_1}, \ \|\dot K_\theta f\|_{H_2}\leq D_1\|f\|_{H_1}$,
	\item $\|(K_\theta - K_{\theta_0})f\|_{S} \leq D_2|\theta-\theta_0|\|f\|_{H}$,
	\item $\|(K_\theta-K_{\theta_0}-(\theta-\theta_0)\dot K_{\theta_0})f\|_{S}\leq D_3|\theta-\theta_0|^2\|f\|_{H},$
	\item $\|\dot K_{\theta_0}(f-f_0)\|_{S}\leq D_4\|f-f_0\|_{H}.$
\end{enumerate}
\end{asspt}

Item (i) simply states that both $K_\theta$ and $\dot K_\theta f$ are uniformly bounded over $\Theta$. Item (ii) is a Lipschitz estimate for $K_\theta f$ at $\theta_0$ as a map from $\Theta$ into $S$. Item (iii) says something similar, but this time about the quality of the linear approximation of $K_\theta-K_{\theta_0}$ at $\theta_0$. Observe that item (iv) can be inferred from items (ii) and (iii) by applying them to $f-f_0$. We still leave it as an assumption for convenience. 

Our first theorem is a semi-parametric BvM for the marginal posterior at $\theta_0$. It is an application of Theoerem 12.9 in \cite{vaartghosal}. To hold, the latter requires in particular that the LAN expansions of the log-likelihood ratios of ``least-favourable transformations" at $(\theta,f)\in\Theta\times H$ be uniformly controlled over sets of posterior mass tending to one. Assumption \ref{asspt1} provides the regularity needed to ensure that the remainder terms in these expansions are of the right order, at least on sets where $\|f\|_{H}$ is bounded and where $\theta, K_{\theta_0}f$ and $\dot K_{\theta_0}f$ are respectively close to $\theta_0, K_{\theta_0}f_0$ and $\dot K_{\theta_0}f_0$. This is why we impose that  $H$-balls of a large enough radius have posterior mass tending to 1, and the reason for our posterior consistency assumptions. Saying that the posterior \text{is consistent at $(\theta_0,f_0)$} with respect to (w.r.t) a semi-metric $d$ means that $\Pi(d((\theta,f),(\theta_0,f_0))<\delta_n|X^{(n)})\to 1$ in $P^n_{\theta_0,f_0}$-probability for some $\delta_n\downarrow 0$ referred to as the \textit{rate of contraction}. Finally, the prior $\pi_f$ on $f$ must exhibit some insensitivity to shifts in the least favourable direction.

\begin{thm}\label{thm: main}
	Let $(\theta_0,f_0)\in\Theta\times H$ be such that $\tilde i_{\theta_0,f_0}>0$. Suppose Assumption \ref{asspt1} is satisfied. Assume that for $M$ large enough, the posterior mass of $\{f: \|f\|_{H}<M\}$ tends to 1. Assume further that the posterior is consistent at $(\theta_0,f_0)$ w.r.t the semi-metric $|\theta-\theta_0| + \|K_{\theta_0}(f-f_0)\|_{H_2}$ and  w.r.t $\|\dot K_{\theta_0}(f-f_0)\|_{H_2}$. Assume as well that there exists a sequence $\gamma_n$ approximating $\gamma_{\theta_0,f_0}$ at a rate $\rho_n\downarrow 0$ w.r.t $\|\cdot\|_{H_1}$. Finally, assume the existence of subsets $\Theta_n\subset\Theta$ and $H_n \subset H$ such that
	$\sqrt n(\Theta_n-\q_0)\uparrow\RR$ and 
	\begin{align}\label{eq: contraction}
		\Pi(\Theta_n\times H_n | X^n)\gaatp 1, 
		\qquad \inf_{\theta\in\Theta_n}\Pi^{\theta=\theta_0}\bigl(H_n + (\theta-\theta_0)\gamma_n | X^n\bigr) &\gaatp 1,\\
		\label{eq: lfdrate}
		\sqrt{n}\sup_{f\in H_n}\bigl|\ip{K_{\theta_0}(f-f_0)}{K_{\theta_0}(\gamma_{\theta_0,f}-\gamma_n)}\bigr|&\to 0,\\
		\label{eq: priorshift}
		\sup_{\theta\in\Theta_n, f\in H_n} \left| \frac{\log(d\pi_{f+(\theta-\theta_0)\gamma_n}/d\pi_f)(f)}{1+n(\theta-\theta_0)^2}\right| &\to 0.
	\end{align}
	Then the BvM holds at $(\theta_0,f_0)$.
\end{thm}

The proof of Theorem \ref{thm: main} can be found in Section \ref{sec: proofs}. Theorem \ref{thm: main} is still in a ``raw" form, in the sense that we would like to translate conditions \eqref{eq: contraction}--\eqref{eq: priorshift} and the posterior contraction conditions into conditions on the operators and on the prior. This way we can have a systematic method of dealing with example applications. In the next section, we will investigate the posterior contraction conditions of Theorem \ref{thm: main} in the case of a (rescaled) Gaussian process prior $\pi_f$. Before that, we  disentangle conditions \eqref{eq: contraction}--\eqref{eq: priorshift} into conditions on the prior and on the rates of contraction in the following corollary.

\begin{cor}\label{cor: main}
	Let $(\theta_0, f_0)\in\Theta\times H$ be such that $\tilde i_{\theta_0,f_0}>0$. Let $\theta_0$ be an interior point of $\Theta$. Suppose Assumption \ref{asspt1} is satisfied. Assume that both the posterior and the posterior given $\theta_0$ contract at $(\theta_0,f_0)$ w.r.t the semi-metrics $|\theta-\theta_0|+\|K_{\theta_0}(f-f_0)\|_{H_2}$ and $\|\dot K_{\theta_0}(f-f_0)\|_{H_2}$ at rates $\eps_n\downarrow 0$ and $\xi_n\downarrow 0$ respectively (with $\eps_n \leq \xi_n$). Similarly assume that the posterior mass of $\{f:\|f\|_{H}<M\}$ goes to 1 for some $M>0$. Assume as well that there exists a sequence $\gamma_n$ approximating $\gamma_{\theta_0,f_0}$ at a rate $\rho_n\downarrow 0$ w.r.t $\|\cdot\|_{H_1}$.  Further assume that for $\eta_n\downarrow 0$, we have:
	\begin{align}
		\Pi(\|K_{\theta}f-K_{\theta_0}f_0\|_{H_2}\leq \eps_n)&\geq e^{-n\eps_n^2/64},\label{eq: KL}\\
		\pi_f\Bigl(\exists s, t\in (-\eps_n,\eps_n): \Bigl| \log\frac{d\pi_{f+(s+t)\gamma_n}} {d\pi_{f+t\g_n}}(f)\Bigr|
		> \eta_n (1+nt^2)\Bigr)&\le e^{-3n\eps_n^2}.\label{eq: priorLikelihoodRatio}
	\end{align}
	Finally assume that $n\eps_n^2\|K_{\theta_0}(\gamma_n-\gamma_{\theta_0,f_0})\|_{H_2}\to 0$, and that either \begin{itemize}
		\item $n\eps_n^2\xi_n^2\to 0$,
		\item or $\theta\mapsto\|K_\theta f\|_{H_2}$ is constant for every $f\in H_1$. 
	\end{itemize}
	Then the BvM holds at $(\theta_0,f_0)$.
\end{cor}
\begin{proof}
	Let $\Theta_n:=(\theta_0-\eps_n,\theta_0+\eps_n)$ and $H_n=H_{n,1}\cap H_{n,2}$ where,
	\[H_{n,1}:= \{f\in H_1: \|K_{\theta_0}(f-f_0)\|_{H_2}< C\eps_n, \|\dot K_{\theta_0}(f-f_0)\|_{H_2} < \xi_n\},\]
	\[H_{n,2}:=\left\{f: \forall t\in (-\eps_n,\eps_n),\ \left| \log{\dfrac{d\pi_{f+t\gamma_n}}{d\pi_f}(f)}\right|\leq \eta_n(1+nt^2)\right\},\]
	for some $C>0$ that will be determined later in the proof. We apply Theorem \ref{thm: main} to $\Theta_n\times H_n$. 
	
	Beginning with condition \eqref{eq: contraction}, we notice that the posterior mass of $\Theta_n\times H_{n,1}$ goes to 1 by the contraction assumption. For $H_{n,2}$, we apply the remaining mass principle (c.f. Theorem 8.20 in \cite{vaartghosal}). In our signal in white noise model \eqref{eq: obs}, it is the case that the Kullback-Leibler divergence $K(P^n_{\theta_0,f_0}, P^n_{\theta,f})$ and the variance $V_{2,0}(P^n_{\theta_0,f_0}, P^n_{\theta,f})$ are equal to 1/2 and 1 times $n\|K_{\theta}f-K_{\theta_0}f_0\|^2_{H_2}$ respectively. It follows from \eqref{eq: KL} that 
	\[\Pi((K\vee V_{2,0})(P^n_{\theta_0,f_0}, P^n_{\theta,f})\leq n\eps_n^2)\geq \Pi(\|K_{\theta}f-K_{\theta_0}f_0\|_{H_2}\leq \eps_n)\geq e^{-n\eps_n^2/64}.\]
	It then suffices to combine this with \eqref{eq: priorLikelihoodRatio} for the remaining mass principle of Theorem 8.20 in \cite{vaartghosal} to be satisfied.  The posterior probability of $H_{n,2}$ therefore tends to 1. 
	
	For the second part of \eqref{eq: contraction}, we observe that if $\theta\in\Theta_n$, then $|\theta-\theta_0|<2\eps_n$. Besides, since $\gamma_n$ converges to $\gamma_{\theta_0,f_0}$ in $H_1$, we have that $\|\gamma_n\|_{H_1}$ is bounded above by some constant $c>0$. By item (i) of Assumption \ref{asspt1}, it follows that $\|K_{\theta_0}((\theta-\theta_0)\gamma_n)\|_{H_2}\leq 2D_1c\eps_n$. Let then $C$ in the definition of $H_{n,1}$ be equal to $4D_1c$. We will now show that $H_{n,1}+(\theta-\theta_0)\gamma_n$ contains the following ball centered at $f_0$:
	\[B^n_{f_0}(C):=\left\{f\in H_1: \|K_{\theta_0}(f-f_0)\|_{H_2}< \frac{C}{2}\eps_n, \|\dot K_{\theta_0}(f-f_0)\|_{H_2} < \frac{C}{2}\xi_n\right\}.\]
	Let then $f\in B^n_{f_0}(C)$. Write $f= f-(\theta-\theta_0)\gamma_n +(\theta-\theta_0)\gamma_n$. Observe that $f-(\theta-\theta_0)\gamma_n$ satisfies for $n$ large enough,
	\begin{align*}
		\|K_{\theta_0}(f-(\theta-\theta_0)\gamma_n-f_0)\|_{H_2}&\leq 2D_1c\eps_n + \|K_{\theta_0}(f-f_0)\|_{H_2} < 4D_1c\eps_n = C\eps_n,\\
		\|\dot K_{\theta_0}(f-(\theta-\theta_0)\gamma_n-f_0)\|_{H_2}& < 2D_1c\eps_n + C\xi_n/2 \leq C\xi_n \text{ (since $\eps_n\leq\xi_n$)}.
	\end{align*}
	This implies that $f\in H_{n,1}$ and thus that $B^n_{f_0}(C)\subseteq H_{n,1}+(\theta-\theta_0)\gamma_n$. The posterior mass of $B^n_{f_0}(C)$ knowing $\theta_0$ goes to 1 by the contraction assumption given $\theta_0$, and hence so does the posterior mass of $H_{n,1}+(\theta-\theta_0)\gamma_n$. Now, for $H_{n,2}+(\theta-\theta_0)\gamma_n$, we first remark that since $d\pi_{f+t\gamma_n}/d\pi_f(f-s\gamma_n) = d\pi_{f+(s+t)\gamma_n}/d\pi_{f+s\gamma_n}(f)$ almost surely, we can write
	\[H_{n,2} + s\gamma_n = \left\{ f: \forall t\in(-\eps_n,\eps_n), \ \left| \log{\dfrac{d\pi_{f+(s+t)\gamma_n}}{d\pi_{f+s\gamma_n}}(f)}\right|\leq \eta_n(1+nt^2) \right\}.\]
	It follows that $\cup_{|s|<\eps_n} (H_{n,2}+s\gamma_n)^C$ is exactly the event in the left hand side of \eqref{eq: priorLikelihoodRatio}, which has probability bounded above by $e^{-3n\eps_n^2}$. We obtain that \begin{align*}
		&\inf_{\theta\in\Theta_n} \Pi^{\theta=\theta_0}(H_{n,2}+(\theta-\theta_0)\gamma_n\mid X^{(n)})\\&\geq \inf_{|s|<\eps_n} \Pi^{\theta=\theta_0}(H_{n,2}+s\gamma_n\mid X^{(n)})\\
		&\geq 1 - \Pi^{\theta=\theta_0}(\cup_{|s|<\eps_n}(H_{n,2}+s\gamma_n)^C\mid X^{(n)})\\
		&\to 1,
	\end{align*}
	again by the remaining mass principle. Condition \eqref{eq: contraction} is thus fully verified, and condition \eqref{eq: priorshift} then directly follows from the definition of $\Theta_n$ and $H_{n,2}$.
	
	Finally, condition \eqref{eq: lfdrate} can be split into two:
	\begin{align}
		\label{eq: lfdrateBis}\sqrt{n}\sup_{f\in H_n}\left| \ip{K_{\theta_0}(f-f_0)}{K_{\theta_0}(\gamma_n-\gamma_{\theta_0,f_0})} \right|&\to 0,\\
		\label{eq: lfdBias} \sqrt{n}\sup_{f\in H_n}\left| \ip{K_{\theta_0}(f-f_0)}{K_{\theta_0}(\gamma_{\theta_0,f}-\gamma_{\theta_0,f_0})} \right|&\to 0.
	\end{align}
	By definition of $H_{n,1}$, \eqref{eq: lfdrateBis} reduces to checking that $\sqrt{n}\eps_n\|K_{\theta_0}(\gamma_n-\gamma_{\theta_0,f_0})\|_{H_2}\to 0$ as in the statement of the corollary. Furthermore, by definition of $\gamma_{\theta_0,f}$ and $\gamma_{\theta_0,f_0}$ (see \eqref{eq: lfd}), we have that \eqref{eq: lfdBias} is equivalent to 
	\[\sqrt{n}\sup_{f\in H_n}\left| \ip{K_{\theta_0}(f-f_0)}{\dot K_{\theta_0}(f-f_0)} \right|\to 0,\]
	which can be obtained either by having $\sqrt{n}\eps_n\xi_n\to 0$, or by having that $\theta\mapsto \|K_{\theta}f\|$ is constant (because then $K_{\theta_0}H_1\perp \dot K_{\theta_0}H_1$). This concludes the proof.
\end{proof}

\section{Re-scaled Gaussian Process Priors}

\subsection{Posterior Contraction with re-scaled Gaussian Priors}

We begin by emphasizing that although Assumption  \ref{asspt1} was written with $\theta_0$ in its different items, it is  intended to be more general and should hold at every value of $\theta_0$, not just the ground truth. We also introduce a new set of assumptions.
\begin{asspt}\label{asspt:stab}
	Let $(\R,\|\cdot\|_{\R})\subseteq H$ be a normed linear space. For $M>0$, let $B_\R(M):=\{f\in\R: \|f\|_\R < M\}$. There exist $\delta>0$ small enough, constants $C_1, C_2>0$, and $\eta\in (0,1]$ such that 
	\begin{enumerate}[label=(\roman*),leftmargin=*, itemsep=0.5ex, before={\everymath{\displaystyle}}]
		\item $\sup_{(\theta,f)\in \Theta\times B_\R(M)}\{|\theta-\theta_0| + \|K_{\theta_0} (f - f_0)\|_{H_2}: \|K_\theta f-K_{\theta_0}f_0\|_{H_2}<\delta\}\leq C_1\delta,$
		\item $\sup_{f\in B_\R(M)}\{\|\dot K_{\theta_0}(f-f_0)\|_{H_2}: \|K_{\theta_0}(f-f_0)\|_{H_2}<\delta\}\leq C_2\delta^\eta.$
		\item Let $(T,\|\cdot\|_T)\supseteq H_1$ be a normed linear space. There exists a constant $C_3>0$ such that for all $f\in H_1$, \[\sup_{\theta\in\Theta} \|K_\theta f\|_{H_2}\leq C_3\|f\|_{T}.\]
	\end{enumerate}
\end{asspt}


Let us give some intuition as to how Assumption \ref{asspt:stab} relates to the conditions on the posterior in Corollary \ref{cor: main} . For sequences $\eps_n, 
\xi_n\downarrow 0$ and $M$ large enough, the BvM in Corollary \ref{cor: main} requires the following three conditions to simultaneously hold for the posterior:
\begin{enumerate}[label=(\alph*),leftmargin=*, itemsep=0.5ex, before={\everymath{\displaystyle}}]
	\item $\Pi((\theta,f)\in\Theta\times H: |\theta-\theta_0| +\|K_{\theta_0}(f-f_0)\|_{H_2}< \eps_n \mid X^{(n)})\to 1$,
	\item $\Pi((\theta,f)\in \Theta\times H: \|\dot K_{\theta_0}(f-f_0)\|_{H_2}<\xi_n\mid X^{(n)})\to 1$.
	\item $\Pi((\theta,f)\in\Theta\times H:\|f\|_{H}<M\mid X^{(n)})\to 1$,
\end{enumerate}
To obtain the contraction condition (a), we will use the standard approach in Bayesian non-linear inverse problems. That is, we first prove contraction at $(\theta_0,f_0)$ at a rate $\eps_n$ in the natural testing semi-metric of our model, \begin{align}\label{eq:testing_metric}
	 d((\theta_1,f_1),(\theta_2,f_2)) = \|K_{\theta_1}f_1 - K_{\theta_2}f_2\|_{H_2},
\end{align}
and then leverage item (ii) of Assumption \ref{asspt:stab} to obtain the desired contraction (a). We will refer to item (ii) as a \textit{(restricted) stability estimate}. In concordance with classical Bayesian non-linear inverse problems' literature (see \cite{Nickl23}), the ``stability estimate" appellation would typically be used to refer to a statement about the injectivity of the map $(\theta,f)\mapsto K_\theta f$ at $(\theta_0,f_0)$, therefore replacing $\|K_{\theta_0}(f-f_0)\|_{H_2}$ in (ii) of Assumption \ref{asspt:stab} by $\|f-f_0\|_{H_1}$ (see for example \cite{magra_VdM_VdV} where this is the case). In here, if we assume invertibility of $K_{\theta_0}$, then item (ii) can be seen as statement on the injectivity at $(\theta_0,K_{\theta_0} f_0)$ of the map $(\theta, K_{\theta_0}f)\mapsto K_\theta K_{\theta_0}^{-1}(K_{\theta_0}f)=K_\theta f$. This is of course a much weaker statement in general.

Concerning the posterior contraction (b) above at $\dot K_{\theta_0} f_0$, we simply build on (a) and use item (ii) of Assumption \ref{asspt:stab} to obtain contraction at a rate $\xi_n = \eps_n^\eta$. The $\eta$ on the right hand side of the inequality being in between 0 and 1 comes from the fact that we would typically expect slower rates in (b) than in (a).

We further note that for both (a) and (b) to be inferred from (i) and (ii) we also need to make sure that the posterior mass of $B_\R(M)$ converges to 1, at least for $M$ large enough. That is, we need $\Pi(B_\R(M)\mid X^{(n)})\to 1$. Observe that since $\R$ is a continuously embedded subspace of $H$, condition (c) will immediately follow from such a statement.  For a Gaussian prior $\pi_f$ on $f$ supported on $H$, this will generally not hold unless $\pi_f$ biases towards functions with small $\R$-norm as $n\uparrow\infty$. Using a re-scaled Gaussian prior supported on $\R$ will ensure that this is the case. We present the prior construction below.
\begin{asspt}[re-scaled Gaussian prior]\label{asspt:prior}
	Let $\pi'_f$ be a base Gaussian prior  supported on $(\R,\|\cdot\|_\R)$, i.e $\pi'_f$ is a Gaussian measure satisfying \[\pi'_f(f\in H: \|f\|_\R<\infty)=1.\] We denote by $\HHH$ its reproducing kernel Hilbert space (RKHS). For a scaling sequence $\tau_n\downarrow 0$, the prior $\pi_f$ on $f$ is the distribution of \[f = \tau_n f', \;\;\; f'\sim \pi_f'.\]  
\end{asspt}

Finally, observe that although the prior on $f$ is Gaussian and $K_\theta$ is linear for every $\theta$, the prior on $K_\theta f$ is non-Gaussian because $\theta$ is random. This is the reason why item (iii) of Assumption \ref{asspt:stab} will be crucial in proving contraction with respect to (w.r.t) the semi-metric $d$ in \eqref{eq:testing_metric}, as we can derive from it and from item (ii) of Assumption \ref{asspt1} the following \textit{Lipschitz condition} $\forall (\theta_1,f_1), (\theta_2,f_2)\in \Theta\times H,$
\begin{align}\label{eq:Lipschitz}
	 \|K_{\theta_1}f_1-K_{\theta_2}f_2\|_{H_2}\lesssim D_2|\theta_1-\theta_2|\|f_1\|_{H}\wedge\|f_2\|_{H} + C_3\|f_1-f_2\|_T.
\end{align}  

The above will allow us to leverage results on Gaussian priors in our context; for instance by linking the small ball probability of $\Pi$ with the ones  of $\pi_\theta$ and $\pi_f$.

\begin{thm}[Contraction in the natural testing metric]\label{thm:contraction}
	Assume that Assumptions \ref{asspt1} and item (iii) of \ref{asspt:stab} on the operators  are satisfied. Let $\eps_n\downarrow 0$ be a polynomial rate such that $n\eps_n^2\uparrow\infty$ and $\eps_n<< n^{-1/4}$. Assume that  the prior $\Pi=\pi_\theta\times\pi_f$ on $(\theta,f)$ is such that $\pi_\theta$ has a positive density w.r.t Lebesgue measure and such that $\pi_f$ satisfies the construction of Assumption \ref{asspt:prior} for $\tau_n = (\sqrt{n}\eps_n^{-1})$. Suppose that $f_0\in\HHH$ and that for $B_\HHH(1):=\{f:\|f\|_\HHH\leq 1\}$, we have:
	\begin{align}\label{eq:covering}
		\log{N(\eps_n, B_\HHH(1), \|\cdot\|_T)}\leq n\eps_n^2.
	\end{align}
	Then the following convergence holds in $P_{\theta_0,f_0}^n$-probability as $n\to\infty$ for some $m>0$: 
	\begin{align}\label{eq: natural_contr}
		\Pi((\theta,f)\in\Theta\times H_n: \|K_\theta f- K_{\theta_0}f_0\|_{H_2} > m\eps_n \mid X^{(n)}) \to 0,
	\end{align}
	where for $M$ large enough,
	\begin{align}\label{eq:sieves_contraction}
		H_n:=\{f = f_1+f_2 : \|f_1\|_T\leq M\eps_n, \ \|f_2\|_\HHH\leq M, \ \|f\|_\R< M\}.
	\end{align}
	The same is true for $\Pi^{\theta=\theta_0}(\cdot \mid X^{(n)}))$.
\end{thm}
The proof of Theorem \ref{thm:contraction} can be found in Section \ref{sec: proofs}. From the fact that the contraction sieves $H_n$ in \eqref{eq:sieves_contraction} are contained in $B_\R(M)$, we can already deduce condition (c). Conditions (a) and (b) are obtained through the following corollary.

\begin{cor}[Decoupled Contraction and in the Derivative]\label{cor:contraction}
	Assume that items (i) and (ii) of Assumptions \ref{asspt:stab} hold and suppose that the conclusions of Theorem \ref{thm:contraction} are true for a polynomial rate $\eps_n$. Then for $H_n$ as in \eqref{eq:sieves_contraction}, the two following convergences hold in $P_{\theta_0,f_0}^n$-probability as $n\to\infty$:
	\begin{align}\label{eq: contraction_in_K}
		\Pi((\theta,f)\in\Theta\times H_n: |\theta-\theta_0| + \|K_{\theta_0}(f-f_0)\|_{H_2} > C_1m\eps_n \mid X^{(n)}) \to 0,
	\end{align} 
	\begin{align}\label{eq: contraction_in_dev}
		\Pi((\theta,f)\in\Theta\times H_n: \|\dot K_{\theta_0}(f-f_0)\|_{H_2} > C_2(C_1m)^\eta\xi_n \mid X^{(n)}) \to 0,
	\end{align} 
where $\xi_n = \eps_n^\eta$, with $\eta\in (0,1]$ coming from item (ii) of Assumption \ref{asspt:stab}. Statements \eqref{eq: contraction_in_K} and \eqref{eq: contraction_in_dev} also hold for the posterior given $\theta_0$.
\end{cor}
\begin{proof}
	We have that item (i) of Assumption \ref{asspt:stab} implies the following set inclusion whenever $m\eps_n\leq \delta$, 
	\begin{align*}
		&\{(\theta, f)\in \Theta\times H_n: |\theta-\theta_0| + \|K_{\theta_0}(f-f_0)\|_{H_2}>C_1m\eps_n\}\\
		&\subseteq \{(\theta,f)\in \Theta\times H_n: \|K_\theta f - K_{\theta_0}f_0\|_{H_2}> m\eps_n\}.
	\end{align*} 
	It is the case that \eqref{eq: contraction_in_K} then immediately follows from \eqref{eq: natural_contr}. Similarly, item (ii) of Assumption \ref{asspt:stab} implies that whenever $C_1m\eps_n\leq\delta$,
	\begin{align*}
		& \{(\theta,f)\in\Theta\times H_n: \|\dot K_{\theta_0}(f-f_0)\|_{H_2}>C_2(C_1m\eps_n)^\eta\}\\
		&\subseteq \{(\theta, f)\in \Theta\times H_n: |\theta-\theta_0| + \|K_{\theta_0}(f-f_0)\|_{H_2}>C_1m\eps_n\}.
	\end{align*}
	It follows that \eqref{eq: contraction_in_K} implies \eqref{eq: contraction_in_dev}.
	
	The same proof holds for $\Pi^{\theta=\theta_0}(\cdot\mid X^{(n)}))$.
\end{proof}

\subsection{A BvM with re-scaled Gaussian Priors}

With contraction results established in the previous subsection, we can now apply Corollary \ref{cor: main} to derive a BvM result when using a re-scaled Gaussian prior satisfying Assumption \ref{asspt:prior}. We first introduce for $\gamma\in H_1$, the decentering function associated to the base prior $\pi'_f$ as:
\[\psi_{\gamma}(\eps):=\inf_{h\in\HHH:\|\gamma-h\|_{H_1}\leq \eps} \|h\|_\HHH^2.\]
For a linear operator $A:H_1\to H_2$, we also consider the following proxy for quantifying the approximation of $A\gamma$ by elements of $\HHH$:
\[\delta^{A}{\gamma}(\eps):=\inf_{h\in\HHH: \|\gamma-h\|_{H_1}\leq \eps} \|A(\gamma-h)\|_{H_2}.\]
Observe that if $A$ is bounded, then we trivially have $\delta^A_\gamma(\eps)\lesssim \eps$. This is the case for $A=K_{\theta_0}$ for instance, in view of item (i) of Assumption \ref{asspt1}. Nevertheless in some situations, we will see that it is possible to leverage smoothing properties of $K_{\theta_0}$ to obtain smaller values for $\delta^{K_{\theta_0}}_{\gamma_{\theta_0, f_0}}(\eps)$, which explains our choice in stating the following BvM theorem.
\begin{thm}\label{thm: main2}
	Let $(\theta_0,f_0)\in\Theta\times H$ be such that $\tilde i_{\theta_0,f_0}$ with $\theta_0$ an interior point. Let $\Pi=\pi_\theta\times\pi_f$ be the prior on $(\theta,f)$ with $\pi_\theta$ admitting a positive density w.r.t the Lebesgue measure. Suppose that Assumptions \ref{asspt1}, \ref{asspt:stab} and \ref{asspt:prior} all hold, with re-scaling $\tau_n=(\sqrt{n}\eps_n)^{-1}$ in Assumption \ref{asspt:prior}. Let $\eps_n=n^{-\tau}$ with $\tau\in (1/4,1/2)$ be such that \eqref{eq:covering} is satisfied. Let also $\xi_n=\eps_n^\eta$ for $\eta$ coming from item (ii) of Assumption \ref{asspt:stab} and consider a sequence $\rho_n\downarrow 0$. Suppose that one of the following four scenarios is true:
	\begin{enumerate}[label=(\roman*),leftmargin=*, itemsep=0.5ex, before={\everymath{\displaystyle}}]
		\item $n\eps_n^2\xi_n^2 + n\eps_n^2\delta^{K_{\theta_0}}_{\gamma_{\theta_0,f_0}}(\rho_n) + n\eps_n^4\psi_{\gamma_{\theta_0,f_0}}(\rho_n)\to 0$.
		\item $n\eps_n^2\xi_n^2\to 0$ and $\gamma_{\theta_0,f_0}\in\HHH$.
		\item $\forall f\in H_1, \ \theta\mapsto \|K_\theta f\|_{H_2}$ is constant and $\ n\eps_n^2\delta^{K_{\theta_0}}_{\gamma_{\theta_0,f_0}}(\rho_n) + n\eps_n^4\psi_{\gamma_{\theta_0,f_0}}(\rho_n)\to 0$.
		\item $\forall f\in H_1, \ \theta\mapsto \|K_\theta f\|_{H_2}$ is constant and  $\gamma_{\theta_0,f_0}\in\HHH$.
	\end{enumerate}
	Then the BvM holds at $(\theta_0,f_0)$.
\end{thm}
\begin{proof}
	We apply Corollary \ref{cor: main}. All the assumptions of Theorem \ref{thm:contraction} are satisfied, so contraction in the correct semi-metrics is obtained by virtue of Corollary \ref{cor:contraction}. Since it contains $B_\R(M)$, the posterior mass of $\{f:\|f\|_{H_1}<M\}$ goes to 1 also by Corollary \ref{cor:contraction}. 
	
	Concerning condition \eqref{eq: KL}, we have already verified it in the proof of Theorem \ref{thm:contraction}. It therefore remains to check condition \eqref{eq: priorLikelihoodRatio}. By definition of $\psi_{\gamma_{\theta_0,f_0}}(\rho_n)$ and $\delta^{K_{\theta_0}}_{\gamma_{\theta_0,f_0}}(\rho_n)$, it is possible to select a sequence $(\gamma_n)\subset\HHH$ such that $\|\gamma_n-\gamma_{\theta_0,f_0}\|_{H_1}\lesssim \rho_n$ with both $\|\gamma_n\|_\HHH^2\asymp \psi_{\gamma_{\theta_0,f_0}}(\rho_n)$ and $\|K_{\theta_0}(\gamma_n-\gamma_{\theta_0,f_0})\|_{H_2}\asymp \delta^{K_{\theta_0}}_{\gamma_{\theta_0,f_0}}(\rho_n)$ being true. Furthermore, by definition of $\pi_f$, its RKHS $\mathcal{H}$ satisfies the following norm relationship: $\|\cdot\|_\mathcal{H} = \sqrt{n}\eps_n\|\cdot\|_\HHH$. Therefore $\gamma_n\in\HHH$ for all $n\in\mathbb{N}$. By Cameron-Martin's formula, we then have that
	\[\dfrac{d\pi_{f+t\gamma_n}}{d\pi_f}(f) = e^{t\|\gamma_n\|_\mathcal{H}U(f) - t^2\|\gamma_n\|_\mathcal{H}^2},\]
	with $U(f)$ some measurable transformation of $f$ with a standard normal distribution. Besides, by the tail bound on the standard normal distribution, we have that $\pi_f(f: |U(f)|>2\sqrt{n}\eps_n)\leq e^{-3n\eps_n^2}$. It follows that \eqref{eq: priorLikelihoodRatio} is verified provided that \[\dfrac{|\log{(d\pi_{f+(s+t)\gamma_n}/d\pi_{f+s\gamma_n})(f)}|}{1+nt^2}\to 0 \ \text{ on }\{f:|U(f)|\leq 2\sqrt{n}\eps_n\},\] for $|s|<\eps_n$. To see this, we apply the before last display twice to obtain that, for $|s|<\eps_n$, we have on the event $\{f:|U(f)|\leq 2\sqrt{n}\eps_n\}$: 
	\begin{align*}
		\dfrac{|\log{(d\pi_{f+(s+t)\gamma_n}/d\pi_{f+s\gamma_n})(f)}|}{1+nt^2} &\leq \dfrac{\|\gamma_n\|_\mathcal{H}|U(f)|}{\sqrt{n}} + \dfrac{2|s|\|\gamma_n\|_\mathcal{H}^2}{\sqrt{n}} + \dfrac{\|\gamma_n\|_\mathcal{H}^2}{n}\\
		&\leq 2\eps_n\|\gamma_n\|_\mathcal{H} + \dfrac{2\eps_n\|\gamma_n\|_\mathcal{H}^2}{\sqrt{n}} + \dfrac{\|\gamma_n\|_\mathcal{H}^2}{n}\\
		&=2\sqrt{n}\eps_n^2\|\gamma_n\|_\HHH + 2\sqrt{n}\eps_n^3\|\gamma_n\|^2_\HHH + \eps_n^2\|\gamma_n\|^2_\HHH,
	\end{align*}
	where we used that $\sqrt{n}t\leq 1+nt^2$. If $\gamma_{\theta_0,f_0}\notin\HHH$, then  the right hand side of the last display goes to 0 provided that $n\eps_n^4\psi_{\gamma_{\theta_0,f_0}}(\rho_n)+n\eps_n^6\psi_{\gamma_{\theta_0,f_0}}(\rho_n)^2 + \eps_n^2\psi_{\gamma_{\theta_0,f_0}}(\rho_n)\to 0$. Now, as the product of the two other terms, the middle term will go to 0 whenever the other two go to 0. Besides, the first term dominates the last one by virtue of $n\eps_n^2\uparrow\infty$. This thus yields items (i) and (iii) of the theorem. 
	
	In the case that $\gamma_{\theta_0,f_0}\in\HHH$, then one can trivially  select $\gamma_n=\gamma_{\theta_0,f_0}$, yielding $\rho_n=0$ for all $n\in\mathbb{N}$. It follows that $\delta^{K_{\theta_0}}_{\gamma_{\theta_0,f_0}}(\rho_n) = 0$ and $\psi_{\gamma_{\theta_0,f_0}}(\rho_n)=\|\gamma_{\theta_0,f_0}\|_\HHH^2$. Realising that $n\eps_n^4\to 0$ by assumption on $\eps_n$, we obtain that items (i) and (iii) respectively imply items (ii) and (iv) in this case.
\end{proof}

\section{Application to Semi-blind Deconvolution}

Consider a family of convolution kernels $\{g_\theta:\theta\in\Theta\}$ indexed by their location $\theta$, with $\Theta\subset (-1/2,1/2)$ compact. We let $g:=g_0$ be the convolution kernel with location 0. We assume that it is a symmetric, square integrable, differentiable and a 1-periodic function. The kernels are identical except for their location which uniquely identifies them, i.e. we have the relationship $g_\theta = g(\cdot - \theta)$. We also define $L^2[-1/2,1/2]$ as the square integrable functions over $[-1/2,1/2]$ extended by 1-periodicity. For $\theta\in\Theta$, we define the convolution operator $K_\theta: L^2[-1/2,1/2]\to L^2[-1/2,1/2]$ as:
\[
	K_{\theta}f(t):=g*f(t-\theta) =\int_{-1/2}^{1/2} f(t-u)g(u-\theta)du.
\]
Because $g$ is differentiable with derivative $g$, we also obtain the following expression for $\dot K_\theta f$,
\[
\dot K_\theta f = -\int_{-1/2}^{1/2} f(t-u)g'(u-\theta)\,du.
\]

Our objective is to recover the location $\theta$ from a single noisy observation of the convoluted signal $K_\theta f$ over the whole domain $[-1/2, 1/2]$. For $W$ a standard Brownian motion, we let $P^{(n)}_{\theta,f}$ be the distribution of $X^{(n)}$ where,
\begin{align}\label{eq: BM_obs}
	dX^{(n)}_t = K_\theta f(t)dt + \dfrac{1}{\sqrt{n}}dW_t.
\end{align}
For $H_1=H_2=L^2[-1/2,1/2]$ and $\dot{W} = dW_t$, this is precisely the observational model \eqref{eq: obs}. This model is a priori not identifiable over all of $L^2[-1,2/2]$ since we have the relationship $K_\theta f = K_\tau h$ for $h(t)= f(t - (\theta-\tau))$. That is, any shifted version of $f$ will yield an equivalent observation with a different location. It is therefore crucial to consider a restricted parameter space for $f$ to pursue any meaningful inference. We thus consider the following two distinct sets of assumptions on the parameter space $H \subset L^2[-1/2,1/2]$.

\begin{model}\label{model: symmetric}
	The parameter space for $f$ is the set of symmetric functions in $L^2[-1/2,1/2]$. Namely,
	\[H =\left\{f\in\L^2[-1/2,1/2]:\ \forall t,\ f(t)=f(-t) \right\}.\]  
\end{model}

\begin{model}\label{model: zeroloc}
	The parameter space for $f$ is the set of functions in $L^2[-1/2,1/2]$ with location zero. Namely,
	\[H = \left\{f\in L^2[-1/2,1/2]: \int_{-1/2}^{1/2}tf(t)dt = 0\right\}.\]
\end{model}

Observe that Model \ref{model: symmetric} is a strict sub-model of Model \ref{model: zeroloc} but exhibits interesting properties which makes it worthwhile to consider on its own. Another way of writing $H$ in Model \ref{model: zeroloc} is as the kernel of the map $f\mapsto \ip{f}{S}$ for $S$ the \textit{saw-tooth} function, which we define as the 1-periodic extension of the identity on $[-1/2,1/2]$. Namely, $S(t) := t$ for $t\in[-1/2,1/2)$ and $S(t+1)=S(t)$ for all $t\in\RR$. The following lemma characterises the information and least favourable directions in both models. 

\begin{lemma}\label{lemma: info_deconv}
	The least favourable direction $\gamma^{(1)}_{\theta, f}$ and corresponding efficient information $\tilde i^{(1)}_{\theta, f}$ in Model \ref{model: symmetric} are as follows:
	\begin{align*}
		\gamma^{(1)}_{\theta, f} = 0, \;\;\;\;\;\;\;\;\;\;\;\;\;\;\;\; \tilde i^{(1)}_{\theta, f} = \|\dot K_\theta f\|_{L^2}^2 = \|g'*f\|^2_{L^2}.
	\end{align*}
	Assume that the mass of $g$ is equal to 1, i.e. $\int_{-1/2}^{1/2}g(t)dt = 1$, and  suppose that $f\in H$ admits a derivative $f'$. Let $A$ be any linear operator such that its adjoint $A^*$ preserves location, i.e. $\ip{A^*h}{S} = \ip{h}{S}$ for all $h\in L^2[-1/2,1/2]$, with $S$ the saw-tooth function. Let also $\lambda = 12\int_{-1/2}^{1/2}f'(t)tdt$. Then $\gamma^{(2), A}_{\theta, f}$ below is a least favourable direction in Model \ref{model: zeroloc}.
	\begin{align*}
		 \gamma^{(2), A}_{\theta, f} = -f' + \lambda AS.
	\end{align*}
	The efficient information in Model \ref{model: zeroloc} is then $\tilde{i}^{(2)}_{\theta, f} = \lambda^2/12$. In particular, if $f(1/2)\neq \int_{-1/2}^{1/2}f(t)dt$ then $i^{(2)}_{\theta,f}>0$.
\end{lemma}
The proof of Lemma \ref{lemma: info_deconv} can be found in Section \ref{sec: proofs}. We remark that in Model \ref{model: symmetric}, we do not lose any information from not knowing $f$, while this is not the case in Model \ref{model: zeroloc}. We also have infinitely many least favourable directions in Model \ref{model: zeroloc}. This means in particular that we can select ``smoother" least favourable directions by leveraging the choice of $A$. This will be useful when we will consider the approximation of $\gamma^{(2),A}_{\theta,f}$ by RKHS elements of a Gaussian prior. Note also that (repeated) convolution with $g$ is an example of a location preserving operator as described in the lemma (this fact will be later proven in Claim \ref{claim: location}).

We are now all set to begin the verification of the necessary conditions to apply the result of Theorem \ref{thm: main2}. We will begin by verifying the operators related Assumptions \ref{asspt1} and \ref{asspt:stab} in the next subsection. We will then specify a Gaussian prior that satisfies Assumption \ref{asspt:prior} and inequality \eqref{eq:covering} before finally deriving BvM results for Models \ref{model: symmetric} and \ref{model: zeroloc}.

In what follows, we will often make use of an alternative representation of $K_\theta$ and $\dot K_\theta$ in the complex exponential basis. Let $\{e_k(\cdot)\}_{k\in\mathbb{Z}}:=\{e^{-2\pi ik\cdot}\}_{k\in\mathbb{Z}}$. For $g_k = \ip{g}{e_k}$ and $f_k=\ip{f}{e_k}$, we have,
\begin{align*}
	K_\theta f &= \sum_{k\in\mathbb{Z}}f_k g_k e^{2\pi i k\theta}e_k,\\
	\dot K_\theta f &= -2\pi i\sum_{k\in\mathbb{Z}}kf_k g_k e^{2\pi ik\theta}e_k.
\end{align*}

\subsection{Regularity and Stability Assumptions}

For $r\geq 0$, consider the Sobolev type spaces $S^r$ and their corresponding square norm:
\begin{align}\label{eq: sobolev}
	S^r:= \left\{ f\in L^2[-1/2,1/2]: \sum_{k\in\mathbb{Z}} |f_k|^2(1+|k|)^{2r} <\infty \right\}, \ \|f\|_{S^r}^2:=\sum_{k\in\mathbb{Z}} |f_k|^2(1+|k|^{2r}).
\end{align}
Note that $\{S^r\}_{r\geq 0}$ are Hilbert spaces and that in particular, $S^0=L^2[-1/2,1/2]$. We also define $S^{-r}$ as the topological dual of $S^r$. The square (dual) norm of $f$ in this space is then equal to \[\|f\|_{S^{-r}}^2 = \sum_{k\in\mathbb{Z}} |f_k|^2(1+|k|)^{-2r},\] further making explicit that $S^{-r}\supset L^2[-1/2,1/2]$ for $r>0$. 

\begin{prop}\label{prop: regularity}
	Let $H_1=H_2=L^2[-1/2,1/2]$ and let $H$ be as in Model \ref{model: symmetric} or Model \ref{model: zeroloc}. Let also $\Theta\subset (-1/2, 1/2)$ be compact. Assume that $\sup_k |k|^{5/2+\delta}|g_k| <\infty$ for some $\delta>0$. Then the convolution operator $K_\theta$ satisfies Assumption \ref{asspt1} with $S=S^r$ for $r\leq 1/2+\delta$. Besides, \eqref{eq: finiteEntropy} is verified for $r>1/2$ and the map $\theta\mapsto \|K_\theta f\|_{L^2}$ is constant for every $f\in H_1$. 
\end{prop}
\begin{proof}
	The proof of Proposition \ref{prop: regularity} is similar to the proof of Proposition 3.2 in \cite{magra_VdV_HvZ}.
\end{proof}

The following proposition summarise the needed assumptions for the full verification of Assumption \ref{asspt:stab} in both Models \ref{model: symmetric} and \ref{model: zeroloc}.
 
\begin{prop}\label{prop: stability}
	Let $\Theta\subset (-1/2,1/2)$ be compact and let $H$ be as in Model \ref{model: symmetric} or \ref{model: zeroloc}. Let $\R = S^\beta\cap H$ for some $\beta\geq 0$. Assume that the coefficients of $g$ satisfy $|g_k|\lesssim |k|^{-\kappa}$ for $\kappa>1$. Let $\chi = \beta+\kappa$ and assume $f_0\in\R$.  Assume further that both $f_0$ and $g$ have mass 1 over their period. Then Assumption \ref{asspt:stab} is verified for $T=S^{-\kappa}$ and $\eta = 1-1/\chi$. 
\end{prop}

\begin{proof}
	In what follows, we will denote by $L^2$ the space $L^2[-1/2,1/2]$ and by $\|\cdot\|$ the $L^2$-norm. By the last statement of Proposition \ref{prop: regularity}, the map $\theta\mapsto \|K_\theta f\|$ is constant. It follows that for any $f\in L^2$, we have
	\begin{align*}
		\sup_{\theta\in\Theta} \|K_\theta f\|^2 = \|g*f\|^2 = \sum_{k\in\mathbb{Z}} |g_k|^2|f_k|^2\lesssim \sum_{k\in\mathbb{Z}} |k|^{-2\kappa}|f_k|^2\asymp \|f\|_{S^{-\kappa}}^2. 
	\end{align*}
	This shows item (iii) of Assumption \ref{asspt:stab} with $T=S^{-\kappa}$.
	
	Assume now that $f\in S^\beta$ such that $\|f\|_{S^\beta} < M$. Letting $h:=f-f_0$, we then have that $\|h\|_{S^\beta}<M+\|f_0\|_{S^\beta}=:M'$. Item (ii) of Assumption \ref{asspt:stab} is then verified by virtue of the next display.
	\begin{align*}
		\|\dot K_{\theta_0} h\|^2 &= 4\pi^2\sum_{k\in\mathbb{Z}} k^2|g_k|^2 |h_k|^2\\
		&\lesssim \sum_{k\in\mathbb{Z}} |k|^{2(1-\kappa)}|h_k|^2\\
		&= \sum_{k\in\mathbb{Z}}\left( |k|^{2-\frac{2\kappa}{\chi}}|h_k|^{\frac{2}{\chi}}\right)\cdot \left(|k|^{-2\kappa\cdot \frac{\chi-1}{\chi}}|h_k|^{2\cdot\frac{\chi-1}{\chi}}\right)\\
		&\leq \left( \sum_{k\in\mathbb{Z}} |k|^{2(\chi-\kappa)}|h_k|^2 \right)^{\frac{1}{\chi}} \cdot \left(\sum_{k\in\mathbb{Z}}|k|^{-2\kappa}|h_k|^2 \right)^{\frac{\chi-1}{\chi}}\text{ (Hölder's inequality)*}\\
		&\asymp \|h\|_{S^\beta}^{2/\chi}\cdot \|K_{\theta_0}h\|^{2\cdot\frac{\chi-1}{\chi}} \text{ (as $\beta=\chi-\kappa$)}\\
		&\leq (M')^{2/\chi}\|K_{\theta_0}h\|^{2\eta}.
	\end{align*}
	*Hölder's inequality was used with $p=\chi, q=\frac{\chi}{\chi-1}$. This is a valid application since $\chi = \beta+\kappa > 1$ by assumption.
	
	We finally verify item (i). We first introduce some notation. Let $h(\cdot-\theta):=g_\theta*f$ and $h_0(\cdot-\theta):=g_\theta*f_0$. Let also $\delta:=\theta_0-\theta$. We want to show that following holds:
	\begin{align}\label{eq:lemmaGOAL}
		\|h(\cdot-\theta)-h_0(\cdot-{\theta_0})\|^2\gtrsim \delta^2 + \|h-h_0\|^2.
	\end{align}
	Observe that,
	\begin{align*}
		&\|h(\cdot-\theta)-h_0(\cdot-{\theta_0})\|^2 \\
		&= \|h(\cdot-\theta)-h_0(\cdot-{\theta}) + h_0(\cdot-\theta)-h_0(\cdot-{\theta_0})\|^2\\
		&=\|h-h_0\|^2+\int (h_0(t)-h_0(t-\delta))^2dt + 2\int (h(t)-h_0(t))(h_0(t)-h_0(t-\delta))\text{d}t\\
		&=: \|h-h_0\|^2 + \|D_\delta\|^2 + 2 \ip{h-h_0}{D_\delta}.
	\end{align*}
	We will make use of the following claim.
	\begin{claim}\label{claim:auxilaryLemma1}
		Let $L^2:=\sum_{k}|h_{0,k}|^2k^2$, then $\|D_\delta\|^2\leq 4\pi^2L^2\delta^2$.
	\end{claim}
	Let us analyse the inner product $2\ip{h-h_0}{D_\delta}$. Since both $h$ and $h_0$ are in $H$, the functions with location 0, we have that $h-h_0$ also has location 0, and so $h-h_0\perp S$, where we recall that $S$ is the sawtooth function equal to the identity on $[-1/2,1/2]$. Defining $\Pi_S D_\delta$ as the projection of $D_\delta$ on the span of $S$, it follows that $2\ip{h-h_0}{D_\delta} = 2\ip{h-h_0}{D_\delta - \Pi_S D_\delta}$. By CS,
	\begin{align*}
		2|\ip{h-h_0}{D_\delta - \Pi_S D_\delta}|
		&\leq 2\|h-h_0\|\cdot\|D_\delta-\Pi_S D_\delta\|\\
		&= 2\left(\sqrt{V}^{-1}\|h-h_0\|\cdot \sqrt{V}\|D_\delta-\Pi_SD_\delta\|\right)\\
		&\leq V^{-1}\|h-h_0\|^2 + V\|D_\delta-\Pi_S D_\delta\|^2,
	\end{align*}
	where $V:=\frac{4\pi^2L^2}{4\pi^2L^2 - 1/24}$. Assume for now that $4\pi^2L^2 > 1/12$, and thus $V>1$.
	Note that $\|D_\delta-\Pi_S D_\delta\|^2 = \|D_\delta\|^2 - \|\Pi_S D_\delta\|^2$, and that $\|\Pi_S D_\delta\|^2=\frac{|\ip{D_\delta}{S}|^2}{\|S\|^2}$. Furthermore, since $\int g = \int f_0 = 1$ and they both have location 0, straightforward computations reveal that $g*f_0(\cdot - c)$ has location $c$ for any $c\in\mathbb{R}$. Hence, $D_\delta$ has location $-\delta$. This yields $|\ip{D_\delta}{S}|^2 = \delta^2$, and as $\|S\|^2=1/12$, we obtain $\|\Pi_S D_\delta\|^2 = \frac{\delta^2}{12}$. It follows that:
	\begin{align*}
		\|h(\cdot-\theta)-h_0(\cdot-{\theta_0})\|^2
		&\geq \|h-h_0\|^2 + \|D_\delta\|^2
		 - \left(V^{-1}\|h-h_0\|^2 + V\|D_\delta\|^2 - V\frac{\delta^2}{12}\right)\\
		&= (1-V^{-1})\|h-h_0\|^2 + \frac{V}{12}\delta^2 - (V-1)\|D_\delta\|^2\\
		&\geq (1-V^{-1})\|h-h_0\|^2 +\left(4\pi^2L^2 - V(4\pi^2L^2 - 1/12) \right)\delta^2,
	\end{align*}
	where the last line follows by Claim \ref{claim:auxilaryLemma1}. Since $V>1$, $1-V^{-1}$ is greater than 0. By assumption that $4\pi^2L^2>1/12$ and by definition of $V$, the coefficient of $\delta^2$ is also positive. Taking the minimum of both coefficients concludes the proof for the case $4\pi^2L^2 > 1/12$. 
	
	If $4\pi^2L^2 \leq 1/12$, then Claim \ref{claim:auxilaryLemma1} implies $\|D_\delta\|^2\leq \delta^2/12$. Let then $V=2$. We obtain:
	\begin{align*}
		\|h(\cdot-\theta)-h_0(\cdot-{\theta_0})\|^2
		&\geq \frac{1}{2} \|h-h_0\|^2 - \|D_\delta\|^2 + \frac{1}{6}\delta^2\\
		&\geq \frac{1}{2} \|h-h_0\|^2 - \frac{1}{12}\delta^2 + \frac{1}{6}\delta^2\\
		&\geq \frac{1}{12} (\|h-h_0\|^2+\delta^2).
	\end{align*}
	We conclude with the proof of Claim \ref{claim:auxilaryLemma1} for completeness:
	\begin{align*}
		\|D_\delta\|^2 &= \|h_0-h_0(\cdot-\delta)\|^2\\
		&= \sum_{k\in\mathbb{Z}} |h_{0,k}\left(1-e^{-i2k\pi\delta}\right)|^2\\
		&= 4\sum_{k\in\mathbb{Z}} |h_{0,k}|^2\sin^2(k\pi\delta)\\
		&\leq 4\pi^2\delta^2\sum_{k\in\mathbb{Z}}|h_{0,k}|^2k^2.
	\end{align*}
\end{proof}


\subsection{BvM for Semi-blind Deconvolution}

We first construct Gaussian series base priors that satisfy Assumption \ref{asspt:prior}. Consider the linear space $H$ from either Models \ref{model: symmetric} or \ref{model: zeroloc} and let $\{h_k\}_k$ be a basis for it; natural choices are the cosine basis for Model \ref{model: symmetric}, and 
the Fourier basis consisting of all cosines together with sine functions modified to be orthogonal to $S$ for Model \ref{model: zeroloc}.
Let then $\beta\geq 0$ and $\alpha>\beta+1/2$. Let also $\sigma_k \asymp |k|^{-\alpha}$ be a sequence of standard deviations, and $Z_k$ be a sequence of independent standard normal. $\pi'_f$ is defined as the distribution of the random series
	\begin{align}\label{eq: prior}
		f' = \sum_{k} \sigma_k Z_k h_k.
	\end{align} 
It is well known that $\pi'_f$ is supported on $S^\beta\cap H$ and so the first part of Assumption \ref{asspt:prior} is verified for $\R=S^\beta \cap H$, which is what we want in view of Proposition \ref{prop: stability}. Furthermore, the RKHS  of $\pi'_f$ is known to be
\[\HHH = \left\{f\in H: \sum_{k}\sigma_k^{-2}|\ip{f}{h_k}|^2 <\infty\right\},\]
which can be shown equal to $S^\alpha\cap H$.  Now, to select the correct re-scaling sequence $\tau_n$, we must first determine the rate $\eps_n$ which should satisfy condition \eqref{eq:covering}. In view of Proposition \ref{prop: stability}, if we assume that $g$ has coefficients decaying at a rate  $|k|^{-\kappa}$, then $T=S^{-\kappa}$. Condition \eqref{eq:covering} can then be rewritten as
\[\log{N(\eps_n, B_{S^\alpha\cap H}(1), \|\cdot\|_{S^{-\kappa}})}\leq n\eps_n^2.\]
The left hand side above is of order $\eps_n^{-1/(\alpha+\kappa)}$ by Lemma B.1 in \cite{gugushviliVdV}. This yields $\eps_n = n^{-(\alpha+\kappa)/(2\alpha+2\kappa+1)}$, and subsequently $\tau_n = n^{1/(4\alpha+4\kappa+2)}$.

All assumptions are therefore checked and we can now derive BvM results for both Models \ref{model: symmetric} and \ref{model: zeroloc}. As a last observation before stating the BvMs, we mention that choosing a large value for $\beta$ is always possible, but only beneficial to obtain a larger $\eta$, which in turn will yield a faster rate $\xi_n=\eps_n^{\eta}$ for contraction in the derivative $\|\dot K_{\theta_0}\cdot\|$. Since this only matters in cases (i) and (ii) of Theorem \ref{thm: main2} when the maps $\theta\mapsto\|K_{\theta}f\|$ are not constant, it will not be useful to consider $\beta>0$ in our examples. We thus select $\beta=0$ in both applications.

\begin{thm}[BvM -- Model 1]\label{thm: model1}
	Let $H$ be as in Model \ref{model: symmetric} and let $\Theta\subset (-1/2,1/2)$ be compact. Assume that the data $X^{(n)}$ in \eqref{eq: BM_obs} is generated according to the pair $(\theta_0, f_0)$ with $\theta_0$ an interior point of $\Theta$, and with $f_0\in S^\alpha\cap H$ with mass equal to 1. Assume also that $g$ has mass 1 and that $|g_k|\lesssim |k|^{-\kappa}$ for $\kappa>5/2$. Let then $\alpha>1/2$ and let $\Pi=\pi_\theta\times\pi_f$ be the prior on $(\theta, f)$ with $\pi_\theta$ having a positive density w.r.t the Lebesgue measure, and with $\pi_f$ the distribution of $n^{1/(4\alpha+4\kappa+2)}f'$ for $f'\sim \pi'_f$, with $\pi'_f$ as in \eqref{eq: prior}. Then the BvM holds at $(\theta_0, f_0)$ if $\|g'*f_0\|^2>0$.
\end{thm}
\begin{proof}
	The last condition ensures that we have positive efficient information (c.f. Lemma \ref{lemma: info_deconv}). Assumptions \ref{asspt1}, \ref{asspt:stab} are verified in view of Propositions \ref{prop: regularity} and \ref{prop: stability} respectively, while Assumption \ref{asspt:prior} and \eqref{eq:covering} are verified in view of the above discussion. Furthermore, by Lemma \ref{lemma: info_deconv} we have that $\gamma_{\theta_0,f_0}=0$, and so it trivially follows that $\gamma_{\theta_0,f_0}\in\HHH=S^\alpha\cap H$. Besides, the last statement of Proposition \ref{prop: regularity} implies that we are in scenario (iv) of Theorem \ref{thm: main2}, which concludes the proof.
\end{proof}

\begin{thm}[BvM -- Model 2]
	Let $H$ be as in Model \ref{model: zeroloc} and let $\Theta\subset (-1/2,1/2)$ be compact. Assume that the data $X^{(n)}$ in \eqref{eq: BM_obs} is generated according to the pair $(\theta_0, f_0)$ with $\theta_0$ an interior point of $\Theta$, and with $f_0\in S^\alpha\cap H$ with mass equal to 1. Assume also that $g$ has mass 1 and that $|g_k|\lesssim |k|^{-\kappa}$ for $\kappa>5/2$. Let then $\alpha>1$ and let $\Pi=\pi_\theta\times\pi_f$ be the prior on $(\theta, f)$ with $\pi_\theta$ having a positive density w.r.t the Lebesgue measure, and with $\pi_f$ the distribution of $n^{1/(4\alpha+4\kappa+2)}f'$ for $f'\sim \pi'_f$, with $\pi'_f$ as in \eqref{eq: prior}. Then the BvM holds at $(\theta_0, f_0)$ if $\int_{-1/2}^{1/2}f_0'(t)tdt \neq 0$
\end{thm}
\begin{proof}
	The last condition ensures that we have positive efficient information in view of Lemma \ref{lemma: info_deconv}. Following the same logic as for the proof of Theorem \ref{thm: model1}, we find ourselves in the situation of scenario (iii) or (iv) depending on whether or not $\gamma_{\theta_0,f_0}\in\HHH$. Since this is a smoothness condition, and since we have flexibility in choosing $\gamma_{\theta_0,f_0}$ (c.f. Lemma \ref{lemma: info_deconv}), we will select it as smooth as possible. Let then $m\in\mathbb{N}$ be such that $g^{(m)}*S\in S^{\alpha-1}\cap H$, where $g^{(m)}$ is convolution with $g$ applied $m$ times. Now, since convolution with $g$ is self-adjoint and location preserving (c.f. Claim \ref{claim: location}), it follows that \[\gamma_{\theta_0,f_0} = -f'+\lambda g^{(m)}*S\] is a valid least favourable direction. Because $f'\in S^{\alpha-1}\cap H$, we obtain that $\gamma_{\theta_0,f_0}\in S^{\alpha-1}\cap H$. Standard computations (as in Example 4.2 in \cite{magra_VdV_HvZ}) then yield that $\psi_{\gamma_{\theta_0, f_0}}(\eps) \lesssim \eps^{-2\alpha/(\alpha-1)}$. On the other hand, using the cut-off sequence approximation to $\gamma_{\theta_0, f_0}$, namely $\gamma^{(n)} := \sum_{|k|>K_n} \ip{\gamma_{\theta_0,f_0,}}{h_k}h_k$, one can show that whenever $\|\gamma^{(n)}-\gamma_{\theta_0,f_0}\|\leq\rho_n$ then $K_n\asymp \rho_n^{-1/(\alpha-1)}$ and also that $\delta^{K_{\theta_0}}_{\gamma_{\theta_0,f_0}}(\rho_n)\lesssim K_n^{-\kappa-\alpha+1}$. We thus obtain that
	\[n\delta^{K_{\theta_0}}_{\gamma_{\theta_0,f_0}}(\rho_n)^2 \lesssim n\rho_n^{\frac{2\kappa+2\alpha-2}{\alpha-1}} \ \text{and } \psi_{\gamma_{\theta_0,f_0}}(\rho_n)\lesssim \rho_n^{-2/(\alpha-1)}.\]
	We want both right hand sides in the above display to be as small as possible, but they have opposite behaviours with respect to $\rho_n$. Setting them both equal gives that the optimal rate is  $\rho_n= n^{-(\alpha-1)/(2\kappa+2\alpha)}$. The rate conditions in item (iii) of Theorem \ref{thm: main2} then reduce to checking that $n\eps_n^4n^{1/(\kappa+\alpha)}\to 0$. Tedious computations show that is satisfied for $\alpha > -\kappa + (3+\sqrt{17})/4$. By assumption that $\kappa>5/2$ and $\alpha>1$, this is always satisfied. This concludes the proof.
	
\end{proof}

\section{Application to Attenuated X-ray Transforms}

The classical X-ray transform of a function $f:\RR^d\to \RR$ (or more generally $f:\RR^d\to \CC$) consists of the collection of all of its line integrals. The inverse problem of recovering $f$ from its X-ray transform has wide applications in tomographic imaging, geophysics and materials science. When $f$ is defined on a manifold, we can consider the more general geodesic X-ray transform. Roughly speaking, the latter records the integrals of $f$ along \textit{geodesics}, which are curves minimising the local distance between two points. For instance, in a Euclidean space, these are just are straight lines, and the transform reduces to the classical X-ray transform. On a curved manifold, however, geodesics may bend according to the metric, and this curvature crucially influences the behavior of the transform.

To formalise this setting, let $(M,g)$ be a Riemannian manifold of dimension $d\geq 2$  and denote by $\partial M$ its boundary. We define the \textit{unit sphere bundle}
\[SM:=\{(x,v): x\in M, g_x(v,v)=1\},\]
whose elements consists of points $x$ in $M$ coupled with a tangent vector $v$ at $x$ with norm $g_x(v,v)$ equal to one. Intuitively, $SM$ collects all possible unit directions one can move in starting from each point in $M$. Each vector $v$ at a point $x$ determines a \textit{unique geodesic} $\gamma_{(x,v)}(t)$ started at $x$ in direction $v$, hence satisfying $\gamma_{(x,v)}(0)=x$ and $\dot\gamma_{(x,v)}(0) = v$. To make notation more compact, we will often refer to the \textit{geodesic flow}
\[\p_t:SM\to SM, \;\;\; \p_t(x,v):=(\gamma_{(x,v)}(t), \dot\gamma_{(x,v)}(t)),\]
which describes how a point in $SM$ moves along geodesics with constant unit speed. We further consider the canonical projection map $\pi:SM\to M, (x,v)\mapsto x$.

The \textit{exit time} $\tau(x,v)$ of the geodesic started at $(x,v)$ is the first time at which it reaches the boundary, namely the first time at which $\pi(\p_t(x,v))\in\partial M$. We will assume that $M$ is a \textit{simple} manifold, which means that (i) it is \textit{nontrapping}, i.e. $\sup_{(x,v)\in SM}\tau(x,v) <\infty$, (ii) its boundary is strictly geodesic convex and (iii) it is free of conjugate points. 

The \textit{boundary bundle} $\partial SM$ is nothing more than the restriction of $SM$ to points $(x,v)$ with $x\in\partial M$. For a point $x\in \partial M$, let $\nu(x)$ be the inner pointing normal at $x$, i.e. towards the inside of the manifold $M$. Three important subsets of $\partial SM$ are the \textit{influx boundary} $\partial_+ SM$ (points pointing towards the inside of $M$), the \textit{outflux boundary} $\partial_-SM$ (points pointing outside of $M$), and the \textit{glancing boundary} $\partial_0 SM$. Formally, 
\begin{align*}
	\partial SM&:= \{(x,v)\in SM: x\in \partial M\},\\
	\partial_{\pm} SM&:= \{(x,v)\in \partial SM: \pm g_x(v,\nu(x))\geq 0\},\\
	\partial_0 SM&:= \{(x,v)\in \partial SM: g_x(v,\nu(x))=0\}.
\end{align*}

We will allow for going ``backwards" along geodesics. Noting that the geodesic started at $(x,-v)$ exits at time $\tau(x,-v)$, we will use the convention that $\p_{-t}(x,v) = \p_{t}(x,-v)$. The geodesic flow $\p_t$ is thus properly defined on $SM$ for $t\in [-\tau(x,-v),\tau(x,v)]$. In particular, for any point $(x,v)\in SM$, we have that
\begin{align*}
	&\p_{-\tau(x,-v)}(x,v) \in \partial_+SM, \\ &\p_0(x,v)=(x,v),\\ &\p_{\tau(x,v)}(x,v)\in \partial_-SM, 
\end{align*} 
Lastly, we let the \textit{footpoint map} $F:SM\to \partial_+SM$ be the map that sends a point $(x,v)$ to its starting point on the influx boundary, i.e.  $F(x,v):=\p_{-\tau(x,-v)}(x,v)$. An important fact that we will use is that $M$ being simple (specifically conditions (i) and (ii) above) implies that both $\tau$ and $F$ are smooth on $SM\backslash \partial_0 SM$. We will denote smooth (complex-valued) functions over $M, SM$ and $\partial_+SM$ by $C^\infty(M), C^\infty(SM)$ and $C^\infty(\partial_+SM)$ respectively.

We are now ready to formally introduce the geodesic X-ray transform of a function $f:M\to \CC$, as the function $K_0f:\partial_+SM\to \CC$ defined for $(x,v)\in \partial_+SM$ as the integral of $f$ along the geodesic started at $(x,v)$. Namely,
\begin{align}\label{eq: xray}
 K_0f(x,v) := \int_0^{\tau(x,v)}f(\pi(\p_t(x,v)))dt.
\end{align}
The map $K_0:C^\infty(M)\to C^\infty(\partial_+SM)$ is injective, and can also be extended to a bounded operator from $L^2(M)$ into $L_2(\partial_+SM)$ (c.f. \cite{Sharafutdinov}), where $L^2(\partial_+SM)$ is understood as the $L^2$-space w.r.t the measure $\mu d\Sigma^{2d-2}$, with $d\Sigma^{2d-2}$ the natural Sasaki area form on $\partial SM$ and $\mu(x.v)=|g_x(v,\nu(x))|.$ The adjoint $K_0^*:L^2(\partial_+SM)\to L^2(M)$ is obtained through the back projection formula,
\begin{align}\label{eq: xray_adjoint}
	K_0^*h(y) = \int_{S_yM}h(F(y,w))dS(w),
\end{align}
with $S_yM$ the space of unit tangent vectors at $y$. 

A natural generalisation of the operator $K_0$ in \eqref{eq: xray} is the \textit{attenuated X-ray transform}, which given an attenuation function $a\in C^\infty(M)$ is defined as
\begin{align}\label{eq: attenuated_xray}
	K_af(x,v):=\int_0^{\tau(x,v)} f(\pi(\p_t(x,v)))e^{-\int_0^t a(\pi(\p_s(x,v)))ds}dt.
\end{align}
The interpretation of $K_a$ is that as the integral of $f$ is computed along the geodesic, it accumulates an attenuation encoded in $M$. The mapping properties of $K_a$ and $K_a^*$ are strongly related to the ones of $K_0$ and $K_0^*$ (see for instance \cite{frigyik}, \cite{monard2019}). 

We will consider constantly attenuated X-ray transforms in this section. For technical ease, we will assume that the attenuation is effective only over a compact set $K$ within the interior of $M$ and that it is purely imaginary. That is, let $\chi\in C^\infty(M)$ be real-valued, supported compactly within the interior of $M$ and satisfying $\chi|_{_K}=1$. For a scalar $\theta$, we denote $K_\theta:= K_{i\theta\cdot\chi}$ for $i$ the imaginary unit. Assuming $\theta\in\Theta\subset (0,2\pi)$ compact, we consider observations in the signal-in white noise model below,
\begin{align}\label{eq: xray_model}
	X^{(n)} = K_\theta f + \frac{1}{\sqrt{n}}\dot W.
\end{align}
For $H_1 = L^2(M)$ and $H_2=L^2(\partial_+SM)$, this is exactly the observational model in \eqref{eq: obs} (with $\dot W$ to be understood as the iso-Gaussian process for $L^2(\partial_+SM)$). The parameter space $H\subset H_1$ for $f$ will be explicitly given later in this section. For now, it suffices to know that the closure of $H$ in $H_1$ is $H_1$. The following lemma characterises the information under certain smoothness conditions. 

\begin{lemma}\label{lemma: xray_info}
	Consider the function $A_\chi:SM\to \CC$, defined as 
	\[A_\chi(y,w) = \int_0^{\tau(y,-w)}\chi(\pi(\p_s(F(y,w))))ds.\]
	Then for any $(\theta,f)\in \RR\times L^2(M)$, we have
	\begin{align}
		K_\theta f(x,v) &= \int_0^{\tau(x,v)}f(\pi(\p_t(x,v)))e^{-i\theta A_\chi(\p_t(x,v))}dt \label{eq: K_theta_xray}\\
		\dot K_\theta f(x,v) &= \int_0^{\tau(x,v)}-if(\pi(\p_t(x,v)))A_\chi(\p_t(x,v))e^{-i\theta A_\chi(\p_t(x,v))}dt,\label{eq: dev_K_theta_xray}
	\end{align}
	where $\dot K_\theta f$ is the derivative of the map $\theta\mapsto K_\theta f$. 
	
	If $f\in C^\infty(K)$, then $K_\theta^*\dot K_\theta f\in C^\infty(M)$. The LFD $\gamma_{\theta,f}$ is then given by 
	\[\gamma_{\theta,f} = (K_\theta^*K_\theta)^{-1}K_\theta^*\dot K_\theta f.\]
	The efficient Fisher information is then given by \eqref{eq: info}.	If $M$ is the unit disk $\DD:=\{x\in\RR^2, \|x\|< 1\}$, then $\gamma_{\theta,f}\in C^\infty(M)$. 
\end{lemma}
The proof of Lemma \ref{lemma: xray_info} can be found in Section \ref{sec: proofs}.


From now on we assume that $M=\DD$. This will be crucial in our analysis to verify the regularity statements of Assumption \ref{asspt1}, as we will leverage forward estimates in \cite{bohr_nickl} that do not have an analog for general simple manifolds at this time. They are based on the non-standard Zernike-Sobolev scale of Hilbert spaces on $\DD$ that we present in the next subsection.

\subsection{Regularity and Stability Assumptions}

For $k\in\mathbb{N}\cup\{0\}$ and $0\leq \ell\leq k$, the Zernike polynomial $Z_{k,\ell}$ is defined for a complex coordinate $z\in\DD$ as
\begin{align*}
	Z_{k,\ell}(z) = \frac{1}{\ell!}\partial_z^\ell \left(z^k\left(\frac{1}{z}-\overline z\right)^\ell\right).
\end{align*}
We denote its normalisation by $\hat Z_{k,\ell}:=Z_{k,\ell}/\|Z_{k,\ell}\|_{L^2(\DD)}$. The Zernike polynomials form a complete orthogonal system of $L^2(\DD)$, and appear in a singular value decomposition of the classical X-ray transform $K_0$ \cite{kazantsev}. For a function $f:\DD\to \mathbb{C}$, let $f_{k,\ell}:=\ip{f}{\hat Z_{k,\ell}}$ be the coefficients of $f$ in this Zernike basis. For $s\geq 0$, we define the Zernike-Sobolev space of order $s$ by
\begin{align}\label{eq: zernike_sobolev}
	\Tilde H^s(\DD)=\left\{ f\in L^2(\DD): \|f\|_{\Tilde H^s(\DD)}^2:=\sum_{k=0}^\infty\sum_{\ell=0}^k (1+k)^{2s}|f_{k,\ell}|^2 <\infty \right\}.
\end{align}
Definition \eqref{eq: zernike_sobolev} can be extended to $s<0$ by duality, with $\Tilde H^{-s}(\DD)=(\Tilde H^s(\DD))^*$ (see section 6.4 of \cite{MonardNicklPternain2021} for details). We further define a scale on $S\DD$ that captures the usual Sobolev regularity in the ``direction" variable and the ``Zernike-regularity" in the space variable:
\begin{align*}
	\Tilde H^s(S\DD) = \left\{u\in L^2(S\DD): \|u\|^2_{\Tilde H^s(S\DD)}:=\sum_{\ell\in\mathbb{Z}}(1+\ell^2)^s\|u_\ell\|_{L^2(\DD)}^2 + \|u_\ell\|_{\Tilde H^s(\DD)}^2<\infty\right\},
\end{align*}
with $u_\ell:\DD\to\mathbb{C}$ defined as the $\ell^{th}$ vertical Fourier mode,
\[u_\ell(x)=\int_0^{2\pi}e^{-i\ell t} u(x, (\cos{t}, \sin{t}))dt.\]
Note that if $f\in \Tilde H^s(\DD)$, we trivially have that $f_\ell =0$ whenever $\ell\neq 0$, and so it follows that  $\|f\|_{\Tilde H^s(S\DD)} = 2\pi\|f\|_{\Tilde H^s(\DD)}$. 
 
The following proposition deals with the regularity statements in Assumption \ref{asspt1}.

\begin{prop}\label{prop: reg_xray}
	Let $H_1= L^2(\DD), H_2 = L^2(\partial_+S\overline\DD)$ and $H = \Tilde H^{\beta}(\DD)$ for $\beta>1$. We have that \eqref{eq: finiteEntropy}  is verified for $S=H^\beta(\partial_+S\overline \DD)$ (the standard Sobolev space of order $\beta$ for functions defined on $\partial_+S\overline\DD)$and that items (i)-(iv) of Assumption \ref{asspt1} are all verified. 
\end{prop}

\begin{proof}
	The influx boundary $\partial_+S\overline\DD$ is two-dimensional. It follows that the $L^2$-metric entropy in of the unit ball in $H^s(\partial_+S\overline\DD)$ is of order $\eps^{-d/s}$. Therefore, the entropy integral on the left hand side of \eqref{eq: finiteEntropy} is finite for $s>1$. This proves that $S=H^\beta(\partial_+S\overline\DD)$ satisfies \eqref{eq: finiteEntropy} by the assumption that $\beta>1$.
	
	To prove the statements of Assumption \ref{asspt1}, we will consider an alternative representation for $K_\theta$. Let $\KK_0$ be the natural extension of the classical X-ray transform $K_0$ \eqref{eq: xray} to functions on $S\DD$. Namely, for $f:S\DD\to\mathbb{C}$,
	\[\KK_0f(x,v) := \int_0^{\tau(x,v)} f(\p_t(x,v))dt.\]
	The attenuated X-ray transform $K_\theta$ then has the alternative representation
	\begin{align}\label{eq: x_ray_alt}
		K_\theta f = R_\theta \KK_0(R_\theta^{-1}f),\;\;\; \text{on }\partial_+S\overline\DD,
	\end{align}
	where $R_\theta:S\overline\DD\to \mathbb{C}$ is an integrating factor for $\theta\chi$, i.e. the solution to the transport equation 
	\[ (X+\theta\chi)R_\theta=0\;\;\; \text{on }S\DD,\]
	with $X=\nu\cdot\nabla_x$ the geodesic vector field. Since $\chi\in C^\infty(\overline \DD)$, we have that $R_\theta\in C^\infty(S\overline\DD)$ (see Remark 4.3 in \cite{bohr_nickl}). Observe that the representation \eqref{eq: x_ray_alt} is also valid for functions $f:S\overline\DD\to\mathbb{C}$. Abusing notation, we obtain from \eqref{eq: K_theta_xray} and \eqref{eq: dev_K_theta_xray}, that $\dot K_\theta f = K_\theta(-ifA_\chi)$. We can therefore write
	\begin{align}\label{eq: alt_dev_xray}
		\dot K_\theta f = -iR_\theta \KK_0(R_\theta^{-1}fA_\chi) \;\;\; \text{on }\partial_+S\overline\DD.
	\end{align}
	
	We now proceed with the proof of Assumption \ref{asspt1}. For the first part of item (i) it suffices to apply Theorem 4.2 in \cite{bohr_nickl} with $s=0$ and $\Phi=\theta\chi\in C^\infty(\overline\DD)$, but we will give a more detailed derivation as it will be useful for proving the other items. Firstly, by Lemma 4.6 in \cite{bohr_nickl}, we have the forward estimate
	\begin{align}\label{eq: x_ray_forward}
		\|\KK_0 u\|_{H^s(\partial_+S\overline\DD)}\lesssim_s \|u\|_{\Tilde H^s(S\DD)},
	\end{align}
	for $s\geq 0$ and $u\in \Tilde H^s(S\DD)$ (with $\lesssim_s$ meaning that the constant depends on $s$). Besides, for $k=2\lceil s/2 \rceil$, we have by Lemma 4.7 in \cite{bohr_nickl} that multiplication is continuous in $C^k(\overline\DD)\times \Tilde H^s (\DD)\to \Tilde H^s(\DD)$ and in $C^k(S\overline\DD)\times \Tilde H^s(S\DD)\to \Tilde H^s(S\DD)$. 
	Since $R_\theta$ and $R_\theta^{-1}$ are both in $C^\infty(S\overline\DD)$, \eqref{eq: x_ray_alt} therefore implies that for $f\in H^s(\DD)$, 
	\begin{align*}
		\|K_\theta f\|_{H^s(\partial_+S\overline\DD)}&\leq \|R_\theta\|_{C^k(\partial_+S\overline\DD)}\cdot \|\KK_0(R_\theta^{-1}f)\|_{H^s(\partial_+S\overline\DD)}\\
		&\lesssim \|R_\theta\|_{C^k(\partial_+S\overline\DD)}\cdot \|R_\theta^{-1}f\|_{\Tilde H^s(S\DD)}\\
		&\leq \|R_\theta\|_{C^k(\partial_+S\overline\DD)}\cdot \|R_\theta^{-1}\|_{C^k(S\overline\DD)} \cdot \|f\|_{\Tilde H^s(S\DD)}\\
		&\leq 2\pi \|R_\theta\|_{C^k(\partial_+S\overline\DD)}\cdot\|R_\theta^{-1}\|_{C^k(S\overline\DD)} \cdot \|f\|_{\Tilde H^s(\DD)}
	\end{align*}
	Now by (4.15) of \cite{bohr_nickl}, we have the estimates 
	\begin{align*}
		\|R_\theta^{\pm 1}\|_{C^k(\partial_+S\overline\DD)}\leq \|R_\theta^{\pm 1}\|_{C^k(S\overline\DD)}\lesssim (1+\|\theta \chi\|_{C^k(\overline\DD)})^k.
	\end{align*}
	Since $\theta\in\Theta$ compact, there exists $\theta^* = \sup_{\theta\in\Theta}|\theta|$. We thus have that the right hand side of the last display is bounded above by $(1+\theta^*\|\chi\|_{C^k(\overline\DD)})^k$ which upon selecting $s=0$ concludes the proof of the first part of item (i). For the second part of item (i), the exact same reasoning using \eqref{eq: alt_dev_xray} instead of \eqref{eq: x_ray_alt} yields
	\[ \|\dot K_\theta f\|_{H^s(\partial_+S\overline\DD)}\lesssim \|R_\theta\|_{C^k(\partial_+S\overline\DD)}\cdot\|R_\theta^{-1}\|_{C^k(S\overline\DD)}\cdot\|A_\chi f\|_{\Tilde H^s(S\DD)}.\] 
	Since $\chi$ is compactly supported in the interior of $\DD$, we have that $A_\chi \in C^\infty(S\overline\DD)$ (see proof of Lemma \ref{lemma: xray_info} for details). We thus have that $\|A_\chi f\|_{\Tilde H^s(S\DD)}\leq\|A_\chi\|_{C^k(S\overline\DD)}\|f\|_{\Tilde H^s(\DD)}$ and we obtain the second part of item (i) by selecting $s=0$ just as before. We also note that we have just proven the stronger statements for any $s\geq 0$ and appropriate constants $C_s$:
	\begin{align}
		\|K_\theta f \|_{H^s(\partial_+S\overline\DD)} &\leq C_s \|f\|_{\Tilde H^s(\DD)},\label{eq: forward_est_xray}\\
		\|\dot K_\theta f \|_{H^s(\partial_+S\overline\DD)} &\leq C_s \|f\|_{\Tilde H^s(\DD)},\label{eq: forward_est_xray_dev}
	\end{align} 
	
	For proving item (ii), we note that by the fundamental theorem of calculus,
	\[(K_\theta-K_{\theta_0})f = \int_{\theta_0}^\theta \dot K_\eta f d\eta.\]
	Therefore by \eqref{eq: forward_est_xray_dev},
	\begin{align*}
		\|(K_\theta-K_{\theta_0})f\|_S &= \|(K_\theta-K_{\theta_0})f\|_{H^\beta(\partial_+S\overline\DD)}\\
		&\leq \int_{\theta_0}^\theta \|\dot K_\eta f\|_{H^\beta(\partial_+S\overline\DD)}d\eta\\
		&\leq C_\beta(\theta-\theta_0)\|f\|_{\Tilde H^\beta(\DD)}\\
		&= C_\beta (\theta-\theta_0)\|f\|_{H}.
	\end{align*}
	
	For item (iii), the integral form of Taylor's remainder theorem yields
	\[(K_\theta - K_{\theta_0} - (\theta-\theta_0)\dot K_{\theta_0})f = \int_{\theta_0}^\theta (\theta - \eta) \ddot K_\eta f d\eta. \]
	Similarly as \eqref{eq: alt_dev_xray}, we can derive the following alternative representation for $\ddot K_\theta f$:
	\[\ddot K_\theta f = -R_\theta\KK_0(R_\theta^{-1}fA_\chi^2).\]
	Noting that $A_\chi^2$ is also in $C^\infty(S\overline\DD)$, we can obtain the following analog estimates to \eqref{eq: forward_est_xray} and \eqref{eq: forward_est_xray_dev}:
	\begin{align}
		\|\ddot K_\theta f \|_{H^s(\partial_+S\overline\DD)} &\leq C_s \|f\|_{\Tilde H^s(\DD)}.\label{eq: forward_est_xray_ddev}
	\end{align}
	The proof of item (iii) is concluded through an application of \eqref{eq: forward_est_xray_ddev}:
	\begin{align*}
		\|(K_\theta - K_{\theta_0} -(\theta-\theta_0)\dot K_{\theta_0})f\|_S &= \|(K_\theta - K_{\theta_0} -(\theta-\theta_0)\dot K_{\theta_0})f\|_{H^\beta(\partial_+S\overline\DD)}\\
		&\leq \int_{\theta_0}^\theta (\theta-\eta)\|\ddot K_\eta f\|_{H^\beta(\partial_+S\overline\DD)}d\eta\\
		&\leq C_\beta (\theta-\theta_0)^2\|f\|_{\Tilde H^\beta(\DD)}\\
		&=  C_\beta (\theta-\theta_0)^2\|f\|_{H}.\\
	\end{align*}
	
	Finally as already mentioned, item (iv) is simply a consequence of items (ii) and (iii).
\end{proof}

To verify Assumption \ref{asspt:stab}, we will need $(\theta_0, f_0)$ to satisfy the following condition.
\begin{condition}\label{condition}
	A pair of parameters $(\theta,f)$ satisfies Condition \ref{condition} if 
	
	(i) $\forall g\in C^\infty(\DD), \ \dot K_\theta f \neq K_\theta g$,
	
	(ii) $\forall (\theta', f')\in \Theta\times H, \ K_\theta f = K_{\theta'}f'\implies (\theta, f)=(\theta', f')$.
\end{condition}
The first part of Condition \ref{condition} will ensure that our model has positive information in view of Lemma \ref{lemma: xray_info}. Reducing it to an assumption on only $(\theta,f)$ can be done by looking into properties of the range of $K_\theta$. Indeed, as seen in the proof of Proposition \ref{prop: reg_xray}, $\dot K_\theta f$ can be written as $K_\theta(-ifA_\chi)$ when considering the natural extension of $K_\theta$ to functions in $C^\infty(SM)$. We are thus requiring that $(\theta,f)$ is chosen so that $K_\theta (-ifA_\chi)$ is not in the range $K_\theta C^\infty(\DD)$. The characterisation of this range is (among other results) given in \cite{assylbekov18}, and can in principle be used to derive when exactly it can be asserted that $\dot K_\theta f$ is not in it. This is quite involved for the context of this paper so we do not pursue it here. Nevertheless, we emphasize that the ``full range" $K_\theta C^\infty(SM)$ is much larger than $K_\theta C^\infty(\DD)$, so it is expected that there exists a nontrivial amount of generic candidate parameter pairs $(\theta, f)$ that will satisfy this part of the condition.

The second part of Condition \ref{condition} is a statement about (global) identifiability of the specific pair $(\theta,f)$. This is often referred to as the identification problem in SPECT. For the closely related problem of identification of the exponential Radon transform, it is well known that identifying the pair $(\theta, f)$ is possible if and only if $f$ is non-radial (\cite{solmon95} or see Theorem 11.12 in the survey \cite{Kuchment}). Other papers deal with this problem with even more general attenuations at times (\citep{Bal,stefanov, Bukhgeim11}) and corroborate the idea that identifying $(\theta,f)$ is possible in a wide variety of scenarios.

The next proposition deals with Assumption \ref{asspt:stab}. 
\begin{prop}\label{prop: stab_xray}
	Assume that $(\theta_0,f_0)$ satisfies Condition \ref{condition}. Let $T=H_1=L^2(\DD)$ and $H_2=L^2(\partial_+S\overline\DD)$. Let also $\R = H = \Tilde H^\beta(\DD)$ for $\beta>1$. Then Assumption \ref{asspt:stab} holds with $\eta=(\beta-1)/\beta$.
\end{prop}

\begin{proof}
	The proof of item (iii) is immediate from \eqref{eq: forward_est_xray} applied with $s=0$. 
	
	To prove item (i) we will first prove that there exists $\mu>0$ and $C>0$ such that for $(\theta,f)\in B_{\R}(M)$,
	\begin{align}
		|\theta-\theta_0| &+ \|f-f_0\|_{L^2(\DD)} <\mu \nonumber  \\ \implies \|K_\theta f - K_{\theta_0}f_0\|_{L^2(\partial_+S\overline\DD)} &\geq C \left(|\theta-\theta_0| + \|K_{\theta_0}(f-f_0)\|_{L^2(\partial_+S\overline\DD)}\right). \label{eq: neighborhood_xray}
	\end{align}
	Assuming \eqref{eq: neighborhood_xray} to be true, we reason by contradiction to prove the lower bound on $\|K_\theta f- K_{\theta_0}f_0\|_{L^2(\partial_+S\overline\DD)}$ everywhere on $\Theta\times B_\R(M)$. By contradiction, there must exists a sequence $(\theta_m, f_m)\subset \Theta\times B_\R(M)$ outside a $\mu$-neighborhood of $(\theta_0, f_0)$ for which we have the convergence
	\[\frac{\|K_{\theta_m}f_m - K_{\theta_0}f_0\|}{|\theta-\theta_0| + \|f_m-f_0\|}\to 0.\]
	Note that the denominator is bounded because $\Theta$ is compact and because $\|f_m - f_0\|_{L^2(\D)} \lesssim \|f_m - f_0\|_{\Tilde H^\beta(\D)}\lesssim C(M)$. It follows that $K_{\theta_m}f_m \to K_{\theta_0} f_0$ in $L^2$. By compactness of $\Theta$, we can assume that $\theta_m\to \theta_1\in\Theta$. Furthermore, by item (ii) of Assumption \ref{asspt1}, we have
	\begin{align*} 
		\|K_{\theta_1}f_m - K_{\theta_0}f_0\| &\leq \|(K_{\theta_1} - K_{\theta_m})f_m\| + \|K_{\theta_m}f_m-K_{\theta_0}f_0\| \\
		&\lesssim |\theta_1-\theta_m| M + \|K_{\theta_m}f_m-K_{\theta_0}f_0\|\\
		&\to 0.
	\end{align*}
	Hence, $K_{\theta_1}f_m\to K_{\theta_0}f_0$ in $L^2$. Besides, by Lemma 6.14 (b) in \cite{MonardNicklPternain2021}, $\Tilde H^\beta(\DD)$ is compactly embedded in $L^2(\DD)$ so that there exists a sub-sequence $f_{m_k}$ converging strongly to a limit $f_1$ in $L^2$. Boundedness of $K_{\theta_1}$ (c.f item (i) of Assumption \ref{asspt1}) implies that $K_{\theta_1}f_{m_k}\to K_{\theta_1}f_1$ in $L^2$, and thus that $K_{\theta_1}f_1 = K_{\theta_0}f_0$. Item (ii) of Condition \ref{condition} then implies that $(\theta_1, f_1)=(\theta_0, f_0)$ which yields a contradiction.
	
	We now prove \eqref{eq: neighborhood_xray}. Assume then that $|\theta-\theta_0|+\|f-f_0\| < \mu$, for $\mu$ to be determined later. By Lemma \ref{lemma: xray_info}, we know that $\gamma_{\theta_0, f_0} \in C^\infty(\DD)$. Item (i) of Condition \ref{condition} then implies that $K_{\theta_0}\gamma_{\theta_0, f_0}\neq \dot K_{\theta_0}f_0$. In view of \eqref{eq: info}, this means that $\tilde i_{\theta_0,f_0}>0$. We now write the difference $K_{\theta}f-K_{\theta_0}f_0$ as
	\[K_\theta f - K_{\theta_0}f_0 = (\theta-\theta_0)\dot K_{\theta_0}f_0 + K_{\theta_0}(f-f_0) + R(\theta,f),\]
	with \[R(\theta,f) = (K_\theta -K_{\theta_0} - (\theta-\theta_0)\dot K_{\theta_0})f_0 + (K_\theta - K_{\theta_0})(f-f_0).\]
	It follows from items (ii) and (iii) of Assumption \ref{asspt1} that for some $D>0$, 
	\[\|R(\theta,f)\|_{L^2(\partial_+S\overline\DD)} \leq D(|\theta-\theta_0|^2 + |\theta-\theta_0|\cdot \|f-f_0\|_{L^2(\DD)}).\]
	We obtain that
	\begin{align*}
		&\|K_\theta f - K_{\theta_0}f_0\|_{L^2(\partial_+S\overline\DD)}\\
		&\geq \|(\theta-\theta_0)\dot K_{\theta_0}f_0 + K_{\theta_0}(f-f_0)\|_{L^2(\partial_+S\overline\DD)} - D(|\theta-\theta_0|^2 + |\theta-\theta_0|\cdot \|f-f_0\|_{L^2(\DD)})\\
		&=|\theta-\theta_0|\left(\left\|\dot K_{\theta_0} f_0 + K_{\theta_0}\left[\frac{f-f_0}{\theta-\theta_0}\right]\right\|_{L^2(\partial_+S\overline\DD)} - D\left[|\theta-\theta_0| +\|f-f_0\|_{L^2(\DD)}\right]\right)\\
		&> |\theta-\theta_0|\left(\tilde i_{\theta_0,f_0}^{1/2} - D\left[|\theta-\theta_0| +\|f-f_0\|_{L^2(\DD)}\right]\right),
	\end{align*}
	where the last inequality stems from the definition of the efficient information $\tilde i_{\theta_0,f_0}$. Let then $C_1 = \tilde i^{1/2}_{\theta_0,f_0}/2$. We therefore have that for $\mu < \tilde i^{1/2} _{\theta_0,f_0}/(2D)$, 
	\begin{align}\label{eq: stab_xray_1}
		\||K_\theta f - K_{\theta_0}f_0\|_{L^2(\partial_+S\overline\DD)}\geq C_1 |\theta-\theta_0|.
	\end{align}
	 Furthermore, by item (ii) of Assumption \ref{asspt1}, we can derive the following inequality
	\begin{align*}
		\|K_\theta f - K_{\theta_0}f_0\|_{L^2(\partial_+S\overline\DD)} \geq \|K_{\theta_0}(f-f_0)\|_{L^2(\partial_+S\overline\DD)} - \tilde D_2M|\theta-\theta_0|,
	\end{align*}
	for an appropriate constant $\tilde D_2$. For simplicity, assume that $\tilde D_2M>1$. Letting $C_2 = (\tilde D_2M)^{-1}$ then yields
	\begin{align}\label{eq: stab_xray_2}
		\|K_{\theta}f-K_{\theta_0}f_0\|_{L^2(\partial_+S\overline\DD)} \geq C_2\|K_{\theta_0}(f-f_0)\|_{L^2(\partial_+S\overline\DD)} - |\theta-\theta_0|.
	\end{align}
	We now select an integer $N$ such that $N-1>2C_1^{-1}$. Applying \eqref{eq: stab_xray_1} $N-1$ times, and \eqref{eq: stab_xray_2} once gives us
	\begin{align*}
		N\cdot \|K_{\theta}f-K_{\theta_0}f_0\| &\geq C_2\|K_{\theta_0}(f-f_0)\| + ((N-1)C_1 - 1)|\theta-\theta_0|\\
		&\geq  C_2\left(\|K_{\theta_0}(f-f_0)\| + |\theta-\theta_0|\right).
	\end{align*}
	Dividing by $N$ both sides of the inequality in the last display therefore yields \eqref{eq: neighborhood_xray} with $C=C_2N^{-1}$.
	
	We finally prove item (ii). Let $h=f-f_0$. The forward estimate \eqref{eq: forward_est_xray_dev} applied for $s=0$ yields
	\[\|\dot K_{\theta_0}h\|^2_{L^2(\partial_+S\overline\DD)}\lesssim \|h\|_{L^2(\DD)} \approx \|h\|_{\tilde H^0(\DD)}.\]
	Let then $t=\beta/(\beta-1)$. From the definition of the Zernike-Sobolev norms, we can use Hölder's inequality to obtain that
	\[\|h\|^2_{\Tilde H^0(\DD)}\lesssim \|h\|^{2/t}_{\Tilde H^{-1/2}(\DD)}\cdot\|h\|^{2(t-1)/t}_{\Tilde H^\beta(\DD)}\]
	The second factor on the right hand side of the above display is bounded by a constant depending on $M$, while for the first factor, we have the stability estimate $\|h\|_{\tilde H^{-1/2}} \leq C(\theta_0\chi, \text{supp}(\chi))\|K_{\theta_0}h\|_{L^2(\partial_+S\overline\DD)}$, by Theorem 4.1 in \cite{bohr_nickl}. Note that the estimate is more generally stated to hold for a smooth enough compactly supported skew-hermitian matrix valued attenuation $\Phi$. Taking $\Phi = i\theta\chi$ in our case certainly verifies these conditions. It follows that for $\eta = 1/t = (\beta-1)/\beta$, we have
	\[\|\dot K_{\theta_0}(f-f_0)\|_{L^2(\partial_+S\overline\DD)}\lesssim \|K_{\theta_0}(f-f_0)\|_{L^2(\partial_+S\overline\DD)}^\eta.\]
\end{proof}


\subsection{BvM for Constantly Attenuated X-ray transforms}

We first construct a prior that satisfies Assumption \ref{asspt:prior}. Let $\pi'_f$ be a series prior of regularity $\alpha>0$. generated using the basis of Zernike polynomials. Namely, let $\pi'_f$ be the distribution of
\begin{align}\label{eq: prior_xray}
	f' = \sum_{k=0}^\infty\sum_{\ell=0}^k \sigma_kW_{k,\ell}Z_{k,\ell},
\end{align}
with $\sigma_k \asymp (1+k)^{-\alpha}$ and $W_{k,\ell}$ i.i.d standard normal random variables. By imposing that $\alpha>\beta+1\ (=\beta + d/2)$, we have that $\pi'_f$ is supported on $\Tilde H^\beta(\DD)=\R$. Furthermore, the RKHS of $\pi'_f$ is known to be
\[\HHH = \left\{ f\in \Tilde H^\beta(\DD): \sum_{k=0}^\infty\sum_{\ell=0}^k \sigma_k^{-2}|f_{k,\ell}|^2 <\infty \right\}, \]
which is essentially just $\Tilde H^\alpha(\DD)$. Now, we have already seen in Proposition \ref{prop: stab_xray} that Assumption \ref{asspt:stab} is satisfied with $T=L^2(\DD)$. Thus, to verify \ref{eq:covering}, we need to find $\eps_n$ for which
\[\log N(\eps_n, B_{\Tilde H^\alpha(\DD)}(1), \|\cdot\|_{L^2(\DD)}) \leq n\eps_n^2.\]
By Lemma B.1 in \cite{gugushviliVdV}, the left hand side above is of order $\eps_n^{-2/\alpha}$. This yields $\eps_n = n^{-\alpha/(2+2\alpha)}$ and subsequently $\tau_n = n^{1/(4\alpha+4)}$. Note also that $\eps_n << n^{-1/4}$ since $\alpha>3/2$.  We now have all ingredients to apply Theorem \ref{thm: main2} and obtain a BvM for the attenuation $\theta$. 

\begin{thm}\label{thm: BvM_xray}
	Let $\alpha>\beta+1>2, \ H=\Tilde H^\beta(\DD),$ and $\Theta$ be a compact subset of $\RR$. Assume that the data $X^{(n)}$ in \eqref{eq: xray_model} is generated according to a true pair $(\theta_0,f_0)$, with $\theta_0$ an interior point of $\Theta$ and $f_0\in C^\infty(K)$ with $K\subset\DD^{int}$ compact. Assume further that $(\theta_0, f_0)$ satisfies Condition \ref{condition}. Let $\Pi=\pi_\theta\times\pi_f$ be the prior on $(\theta,f)$ with $\pi_\theta$ having a positive density w.r.t the Lebesgue measure, and with $\pi_f$ the distribution of $n^{1/(4+4\alpha)}f'$ for $f'\sim \pi'_f$, with $\pi'_f$ as in \eqref{eq: prior_xray}. Then the BvM holds at $(\theta_0,f_0)$.
\end{thm}
\begin{proof}
	The assumptions on $f_0$ ensure positive information and that $\gamma_{\theta_0, f_0}\in C^\infty(\DD)$ in view of Lemma \ref{lemma: xray_info}. Assumptions \ref{asspt1} and \ref{asspt:stab} are verified in view of Propositions \ref{prop: reg_xray} and \ref{prop: stab_xray} respectively, while Assumption \ref{asspt:prior} is verified in view of the above discussion. Since $\HHH=\Tilde H^\alpha(\DD)\supset C^\infty(\DD)$, we have that $\gamma_{\theta_0, f_0}\in \HHH$. We are thus in scenario (ii) of Theorem \ref{thm: main2}. We therefore need to verify that $n^2\eps_n^{2+2\eta}\to 0$. This can be seen equivalent to the condition $\alpha > (2\beta)/(3\beta-2)$. Observe that for $\beta> 1$, we have that $\beta+1>(2\beta)/(3\beta-2)$. The assumption that $\alpha>\beta+1>2$ is thus enough to conclude the proof. 
\end{proof}

\section{Proofs}\label{sec: proofs}

\subsection{Proof of Theorem \ref{thm: main}}

This is an application of Theorem 12.9 in \cite{vaartghosal}. We apply it with the least favorable transformation $(\theta,f)\mapsto (\theta_0, f+(\theta-\theta_0)\gamma_n)$. Its condition (12.14) is satisfied in view of our assumption \eqref{eq: priorshift}, the posterior consistency
conditions in Theorem 12.9 are also copied in our conditions, and we need only verify (12.13) in \cite{vaartghosal}. 
By our assumptions on the posterior, we may assume that the sets $\Theta_n\times H_n$ are contained in the following set for $M>0$ large enough and for $\eps_n,\xi_n\downarrow 0$ the rates of contraction of the posterior w.r.t $|\theta-\theta_0| + \|K_{\theta_0}(f-f_0)\|_{H_2}$ and $\|\dot K_{\theta_0}(f-f_0)\|_{H_2}$ respectively:
\[\{(\q,h)\in \Theta\times H: |\q-\q_0|<\eps_n,\|K_{\theta_0}(f-f_0)\|_{H_2}<\eps_n, \|\dot K_{\theta_0}(f-f_0)\|_{H_2}<\xi_n, \|f\|_{H}<M\}.\]  
In what follows, we will often drop the subscripts $H_1$ and $H_2$ in the norms' notation when it is clear in context.

Straightforward computations using \eqref{eq: likelihood} yield
\[
\log \frac{{dP_{\theta,f}^n}}{dP_{\theta_0, f+(\theta-\theta_0)\gamma_n}^n} 
= \sqrt{n}(\theta-\theta_0) G_{\theta_0}(f,\gamma_n) - \frac{n}{2}|\theta-\theta_0|^2 \tilde i_{\theta_0, f}(\gamma_n) + R_n(\theta,f),
\]
where 
\begin{align*}
	G_{\theta_0}(f,g)&= \ip{\dot W}{\dot K_{\theta_0}f - K_{\theta_0} g},\\
	\tilde i_{\theta_0, f}(g)&= \|\dot K_{\theta_0}f\|^2 - \|K_{\theta_0}g\|^2,\\
	R_n(\theta,f) &= \sqrt{n}\ip{\dot W}{(K_{\theta}-K_{\theta_0} - (\theta-\theta_0)\dot K_{\theta_0})f}\\
	&\qquad + n(\theta-\theta_0)\ip{K_{\theta_0}(f-f_0)} {K_{\theta_0}(\gamma_n-\gamma_{\theta_0,f})} \\
	&\qquad- n\ip{(K_{\theta}-K_{\theta_0} - (\theta-\theta_0)\dot K_{\theta_0})f}{K_{\theta_0}(f-f_0)}\\
	&\qquad +\frac{n}{2}(\theta-\theta_0)^2\|\dot K_{\theta_0}f\|^2- \frac{n}{2}\|(K_{\theta}-K_{\theta_0})f\|^2\\
\end{align*}
Therefore, condition (12.14) of Theorem~12.9 in \cite{vaartghosal} is satisfied, with $\tilde G_n:=G_{\theta_0}(f_0,\gamma_{\theta_0,f_0})$
and $\tilde i_n=\tilde i_{\q_0,f_0}$, if 
\begin{align}
	\sup_{(\theta,f)\in \Theta_n\times H_n} \frac{\sqrt{n}(\theta-\theta_0)| G_{\theta_0}(f, \gamma_n)-G_{\theta_0}(f_0,\gamma_{\theta_0,f_0})|}{1+n(\theta-\theta_0)^2} &\gaatp 0,\label{eq: LANapprox}\\
	\sup_{(\theta,f)\in \Theta_n\times H_n} \frac{n(\theta-\theta_0)^2|\tilde i_{\theta_0,f} (\gamma_n)-\tilde i_{\theta_0, f_0}|}{1+n(\theta-\theta_0)^2} &\to 0,\label{EqContFI}\\
	\sup_{(\theta,f)\in \Theta_n\times H_n} \frac{R_n(\theta,f)}{1+n(\theta-\theta_0)^2} &\gaatp 0.\label{eq: remaindercondition}\\
\end{align}
The Bernstein-von Mises theorem is then satisfied at $(\q_0,f_0)$ with 
$\Delta_{\theta_0, f_0}^n:= \tilde i_{\theta_0,f_0}^{-1}G_{\theta_0}(f_0, \gamma_{\theta_0,f_0})$, which is exactly $N(0, \tilde i_{\theta_0,f_0}^{-1})$
distributed.

Starting with \eqref{eq: LANapprox}, we use the bound $\sqrt{n}|\theta-\theta_0|/(1 + n(\theta-\theta_0)^2) \leq 1$. The  left hand side can then be bounded by, \[\sup_{f\in H_n}\left|\ip{\dot W}{\dot K_{\theta_0}(f-f_0)}\right| + \left|\ip{\dot W}{K_{\theta_0}(\g_n - \g_{\theta_0, f_0})}\right|.\]
There is no dependence on $f$ in the second term, it is simply a mean zero Gaussian random variable with bounded variance $\rho_n^2\downarrow 0$ by assumption. It therefore converges to 0 in probability. The first term is a centered Gaussian process indexed by $\D_n = \{\dot K_{\theta_0}(f-f_0): f\in H_n\}$. By item (iv) of Assumption \ref{asspt1}, the $S$-norm of elements of $\D_n$ are bounded by \[D_4\|f-f_0\|_{H_1}\leq D_4(\|f\|+\|f_0\|)\leq D_4(M+\|f_0\|),\] where the last inequality follows by our initial assumption on $H_n$. The $H_2$-diameter of $\D_n$ is furthermore bounded by $2\xi_n$, by this same assumption. $\D_n$ is therefore contained in a multiple of the unit ball in $S$ and its $H_2$-diameter is shrinking to 0. By Dudley's bound, we have,
\[\mathbb{E}\sup_{x\in \D_n}|\ip{\dot W}{x}|\lesssim \int_0^{\xi_n} \sqrt{\log N(\xi, \{x: \|x\|_{S}\le 1\}, \|\cdot\|_{H_2})}\,d\xi.\]
By \eqref{eq: finiteEntropy}, the above integral is bounded above by a multiple of $\xi_n$, and hence goes to 0. \eqref{eq: LANapprox} then follows.

To bound \eqref{EqContFI}, it suffices to show that $\|\dot K_{\theta_0}f\|^2 - \|\dot K_{\theta_0}f_0\|^2$ and $\|K_{\theta_0}\g_n\|^2 - \|K_{\theta_0}\g_{\theta_0,f_0}\|^2$ uniformly go to 0 on $H_n$. For the first quantity, we have over $H_n$ that,
\begin{align*}
	\|\dot K_{\theta_0}f\|^2 - \|\dot K_{\theta_0}f_0\|^2&=\ip{\dot K_{\theta_0}(f-f_0)}{\dot K_{\theta_0}(f+f_0)}\\
	&\leq \|\dot K_{\theta_0}(f-f_0)\|\|\dot K_{\theta_0}(f+f_0)\| \\
	&\leq \xi_n(D_1(M+\|f_0\|))\\
	&\to 0 \text{ (as }n\to\infty),
\end{align*}
with the last inequality following from item (i) of Assumption \ref{asspt1} and because we are on $H_n$. The second quantity can be bounded similarly again by item (i), but this time using that $\|\gamma_n-\gamma_{\theta_0,f_0}\| < \rho_n$.

It remains to show \eqref{eq: remaindercondition}. We deal with each term in the expression of $R_n(\theta,f)$ separately.

We begin with the first term by bounding $\sqrt{n}/(1+n(\theta-\theta_0)^2)$ by $1/|\theta-\theta_0|$. This way we reduce the problem to bounding the supremum of $\left|\ip{\dot W}{g}\right|$ for $g\in\mathcal{G}_n$ where,
\[\mathcal{G}_n:=\Bigl\{ g_{\q,f}:=\frac{(K_\theta-K_{\theta_0}-(\theta-\theta_0)\dot K_{\theta_0}) f}{\theta-\theta_0}: 
(\theta, f) \in \Theta_n\times H_n \Bigr\}.\]
It follows from item (iii) of Assumption \ref{asspt1} that the $S$-norm of elements of $\mathcal{G}_n$ is bounded above by $|\theta-\theta_0|\|f\| < \eps_n M$. The same is true for the $H_2$-norm. In particular, this means that $\mathcal{G}_n$ is contained in a multiple of the  unit ball in $S$ and has a shrinking $H_2$-diameter. By the same argument as in the proof of \eqref{eq: LANapprox} involving Dudley's bound, we thus have that $\sup_{g\in\mathcal{G}_n}\left|\ip{\dot W}{g}\right|$ goes to 0 in probability.

The second term is bounded directly using the bound $n|\theta-\theta_0|/(1+n(\theta-\theta_0)^2)\leq \sqrt{n}$ in combination with \eqref{eq: lfdrate}.

For the third term, we use Cauchy-Schwartz in combination with item (iii) of Assumption \ref{asspt1} and the fact that $\|K_{\theta_0}(f-f_0)\|<\eps_n$. This yields the bound $D_3M\eps_n$ for its ratio with $1+n(\theta-\theta_0)^2$, which indeed goes to 0.

Finally, we rewrite the fourth term as an inner product and bound it using items (i)-(iii) of Assumption \ref{asspt1}, namely,
\begin{align*}
	&-\frac{n}{2}\ip{(K_\theta-K_{\theta_0}-(\theta-\theta_0)\dot K_{\theta_0})f}{(K_{\theta}-K_{\theta_0})f + (\theta-\theta_0)\dot K_{\theta_0}f}\\
	&\leq \frac{n}{2}D_3|\theta-\theta_0|^2\|f\|\left(D_2|\theta-\theta_0|\|f\| + D_1|\theta-\theta_0|\|f\|\right)\\
	&\lesssim n(\theta-\theta_0)^2|\theta-\theta_0|\\
	&\leq  n(\theta-\theta_0)^2\eps_n\\
	& = o(1+n(\theta-\theta_0)^2).
\end{align*}

\subsection{Proof of Theorem \ref{thm:contraction}}

The intrinsic metric for the white noise model \eqref{eq: obs} is 
\[d((\theta_1,f_1),(\theta_2,f_2))=\|K_{\theta_1}f_1-K_{\theta_2}f_2\|_{H_2}.\]
By a straightforward adaptation of Theorem 8.31 in \cite{vaartghosal}, contraction at $(\theta_0, f_0)$ in the metric $d$ (i.e. \eqref{eq: contraction_in_K}) on $H_n$ will follow if we can show that,
\begin{align}
	\label{eq:cc1}\Pi((\theta,f): d((\theta,f),(\theta_0,f_0))\leq \eps_n) & \ge e^{-n\eps_n^2/64},\\
	\label{eq:cc2}\sup_{\epsilon>\eps_n}\log N\bigl(\eps/8, \{((\theta,f)\in \Theta\times H_n: d((\theta,f),(\theta_0,f_0)) \le \eps\}, d\bigr) & \le n\eps^2_n,\\
	\label{eq:cc3} \Pi((\Theta\times H_n)^C)&\leq e^{-n\eps_n^2}.
\end{align}

We prove each item separately beginning with \eqref{eq:cc1}. Item (iii) of Assumption \ref{asspt:stab} in combination with item (ii) of Assumption imply the Lipschitz condition \eqref{eq:Lipschitz}, which we recall here for $\theta_1,\theta_2\in\Theta, f_1,f_2\in H$:
\begin{align*}
	d((\theta_1,f_1),(\theta_2,f_2))=\|K_{\theta_1}f_1- K_{\theta_2}f_2\|_{H_2} &\leq D_2|\theta_1-\theta_2|\|f_1\|_{H}\wedge\|f_2\|_{H} \\ &+ C_3\|f_1-f_2\|_T.
\end{align*}
Let then $C_0:= D_2\|f_0\|_{H}\vee C_3$. We have,
\begin{align*}
	&\Pi((\theta,f): d((\theta,f),(\theta_0,f_0))\leq \eps_n)\\
	&\geq \Pi((\theta,f): |\theta-\theta_0| + \|f-f_0\|_T \leq \eps_n/C_0)\\
	&\geq \Pi((\theta,f): |\theta-\theta_0|\leq \eps_n/2C_0, \|f-f_0\|_T\leq \eps_n/2C_0)\\
	&=\pi_\theta(|\theta-\theta_0|\leq \eps_n/2C_0)\cdot\pi_f(\|f-f_0\|_T\leq \eps_n/2C_0).
\end{align*}
Since $\pi_\theta$ admits a density that is bounded away from 0, the first factor in the last line of the above display is bounded below by $e^{-cn\eps_n^2}$ for a small enough constant $c$. Furthermore, by Assumption \ref{asspt:prior} on $\pi_f$, we can use the fact that $f_0\in\HHH$ in combination with Anderson's lemma to obtain:
\begin{align*}
	\pi_f(\|f-f_0\|_T\leq\eps_n/2C_0)&\geq e^{-\frac{1}{2}n\eps_n^2\|f_0\|_\HHH^2}\pi_f(f: \|f\|_T\leq \eps_n/2C_0)\\
	&=e^{-\frac{1}{2}n\eps_n^2\|f_0\|_\HHH^2}\pi_f'(f': \|f'\|_T\leq \sqrt{n}\eps_n^2/2C_0).
\end{align*}
We now define for $\eps>0$ the small ball exponent of $\pi'_f$ by \[\phi_0(\eps):=-\log{\pi'_f(f':\|f'\|_T\leq\eps)}.\]
In view of the before last display, what we want to show is therefore that $\phi_0(\sqrt{n}\eps_n^2)\lesssim n\eps_n^2$. Since we assumed $\eps_n$ was polynomial, we can write $\eps_n \asymp n^{-\tau}$ for $\tau\in(1/4,1/2)$. Using condition \eqref{eq:covering}, straightforward computations yield
\[\mathcal{H}(\eps):=\log{N(\eps,\HHH_1,\|\cdot\|_T)}\leq \eps^{-\frac{1-2\tau}{\tau}}.\] An application of Kuelbs-Li-Linde's theorem in $T\supseteq H_1$ supporting $\pi'_f$ (see for instance Theorem 6.2.1 in \cite{gine_nickl}) then gives that $\phi_0(\eps)\lesssim \eps^{-\frac{2-4\tau}{4\tau-1}}$, because $4\tau>1$. It follows that \[\phi_0(\sqrt{n}\eps_n^2) = \phi_0(n^{\frac{1-4\tau}{2}}) \lesssim n^{-\frac{4\tau-1}{2}\cdot\frac{4\tau-2}{4\tau-1}} = n^{1-2\tau} \asymp n\eps_n^2.\]

We can now proceed with the proof of \eqref{eq:cc2}. This can be reduced to showing that for all $m = m(M)$ large enough, 
\[\log N(m\eps_n, \Theta\times H_n, d) \leq n\eps_n^2.\]
Let $C_0':= D_2M\vee C_3$. By the Lipschitz condition \eqref{eq:Lipschitz} and the definition of $H_n$, we can derive the following upper bound for all $m$
\[\log{N(m\eps_n, \Theta\times H_n, d)}\leq \log{N(m\eps_n/2C_0', \Theta, |\cdot|)} + \log{N(m\eps_n/2C_0', H_n, \|\cdot\|_T)}. \]
The first term is bounded by a term of order $\log{\eps_n^{-1}}\asymp \tau \log(n) \leq n^{1-2\tau} = n\eps_n^2$, if $m$ is taken large enough. To bound the second term, it suffices to realise that we can cover $H_n$ with $\|\cdot\|_T$-balls of radius $m\eps_n/2C_0'$ by covering $M\HHH_1$ with $\|\cdot\|_T$-balls of radius $m\eps_n/4C_0'$ if $m/2C_0' > M$. We are thus left with showing that \[\log N(m\eps_n/4C_0', M\HHH_1, \|\cdot\|_T)\leq n\eps_n^2,\] which follows from condition \eqref{eq:covering} for $m>4C_0'$.
 
We finally prove \eqref{eq:cc3}.  Observe that $\Pi((\Theta\times H_n)^C) = \pi_f(H_n^C)$. We will first establish that \[\pi_f(\|f\|_\R> M)\leq e^{-An\eps_n^2},\] for some constant $A>1$. To do so we observe that since $\R$ is separable, its norm can be respresented as $\|f\|_\R = \sup_{\ell\in \mathcal{L}} |\ell(f)|$, for $\mathcal{L}$ a countable family of linear forms on $\R$. This means that if $f'\sim\pi'_f$, then $\{\ell(f'):\ell\in\mathcal{L}\}$ is a centered Gaussian process with countable index set that is supported on $\R$ by Assumption \ref{asspt:prior}, i.e. that $\|f'\|_\R = \sup_{\ell\in\mathcal{L}}|\ell(f')\|<\infty$ almost surely. Theorem 2.1.20 in \cite{gine_nickl} can then be applied to show that $\mathbb{E}[\|f'\|_R] <\infty$ and that
\begin{align*}
	\pi_f(\|f\|_\R) &= \pi'_f(\|f'\|_\R > M\sqrt{n}\eps_n)\\
	&\leq \pi'_f(\|f'\|_\R - \mathbb{E}[\|f'\|_\R]> M\sqrt{n}\eps_n/2)\\
	&\leq e^{-cM^2n\eps_n^2} \text{ (for some $c>0$)}\\
	&\leq e^{-An\eps_n^2},
\end{align*}
for $M$ large enough since $\sqrt{n}\eps_n\uparrow\infty$. It now remains to show that:
\begin{align*}
	&\pi_f(f = f_1+ f_2: \|f_1\|_T\leq M\eps_n, \|f_2\|_\HHH\leq M) \\ 
	&= \pi_f'(f'=f_1'+f_2':\|f_1'\|_T\leq M\sqrt{n}\eps_n^2, \|f_2'\|_\HHH \leq M\sqrt{n}\eps_n)\\
	&\geq 1 - e^{-An\eps_n^2}.
\end{align*}
Similar to the proof of \eqref{eq:cc1} above, we may obtain using Kuelbs-Li-Linde and condition \eqref{eq:covering}, that
\[ -\log{\pi'_f(f':\|f'\|_T \leq M\sqrt{n}\eps_n^2)}= \phi_0(M\sqrt{n}\eps_n^2) \asymp \phi_0(Mn^{\frac{1-4\tau}{2}})\leq cM^{-\frac{2-4\tau}{4\tau-1}}\sqrt{n}\eps_n^2,\]
for some $c>0$. Taking $M>(A/c)^{-\frac{4\tau-1}{2-4\tau}}$ then implies $\phi_0(M\sqrt{n}\eps_n^2)\leq An\eps_n^2$. Furthermore, it is well known that for $\Phi$ the cumulative distribution function of a standard normal, we have
\[A_n:=-2\Phi^{-1}(e^{-An\eps_n^2})\simeq \sqrt{An}\eps_n.\]
By taking $M>\sqrt{A}$ and subsequently using Borell's inequality, we may finish the proof as follows: 
\begin{align*}
	&\pi_f'(f'=f_1'+f_2':\|f_1'\|_T\leq M\sqrt{n}\eps_n^2, \|f_2'\|_\HHH \leq M\sqrt{n}\eps_n)\\
	&\geq  \pi_f'(f'=f_1'+f_2':\|f_1'\|_T\leq M\sqrt{n}\eps_n^2, \|f_2'\|_\HHH \leq A_n)\\
	&\geq \Phi(\Phi^{-1}(\pi_f'(f': \|f'\|_T\leq M\sqrt{n}\eps_n^2)) + A_n)\\
	&= \Phi(\Phi^{-1}(e^{-\phi_0(\sqrt{n}\eps_n^2)})+A_n)\\
	&\geq \Phi(\Phi^{-1}(e^{-An\eps_n^2}) + A_n)\\
	&=\Phi(-\Phi^{-1}(e^{-An\eps_n^2}))\\
	&=1-e^{-An\eps_n^2}.
\end{align*}

The statement about the posterior given $\theta_0$, i.e. $\Pi^{\theta=\theta_0}(\cdot \mid X^{(n)})$ can be proven similarly, omitting the $\theta$-parts of the proof.

\subsection{Proof of Lemma \ref{lemma: info_deconv}}

We begin by observing that since $f\in H$ admits a derivative $f'$, an alternative way of writing $\dot K_\theta f$ is as $K_\theta (-f') = -g*f'(\cdot-\theta)$. From the discussion in subsection 2.1, we know that the least favourable direction $\gamma_{\theta,f}$ is such that $-K_\theta \gamma_{\theta, f}$ is the projection of $-\dot K_\theta f$ onto the closure of $K_\theta H$. 
	
In Model \ref{model: symmetric}, $H$ is the symmetric (even) functions. It follows that $f'$ is odd in this case. Since $g$ is even, convolution with $g$ preserves parity, hence $g*f'$ is odd. It follows that $\dot K_\theta f \perp K_\theta H$ in Model \ref{model: symmetric}, and hence that $\gamma^{(1)}_{\theta, f} = 0$. The expression for the efficient information then stems from \eqref{eq: info}. 
	
To find the least favourable direction in Model \ref{model: zeroloc} we proceed similarly. We will be making use of the following claim that is proved at the end. 
\begin{claim}\label{claim: location}
	Convolution with $g$ is a location preserving operation. Furthermore, for every $\theta, \ K_\theta^* = K_{-\theta}$. In particular, convolution with $g$ is a self-adjoint operation.
\end{claim}

Now, in view of what was stated in the beginning of the proof, $\gamma^{(2),A}_{\theta, f}$, hereby referred to as $\gamma$ for ease of notation, must be in $H$ and must satisfy:
	
\[
\ip{\dot{K}_{\theta} f-K_{\theta}\gamma}{K_{\theta} h} = 0 \;\;\; \forall h\in H.
\]
That $\gamma\in H$ is readily verified by definition of $\lambda$ and the location preserving property of $A^*$ since
\begin{align*}
	\ip{\gamma}{S} &= \lambda \ip{AS}{S} - \ip{f'}{S}\\
	 &= 12\ip{f'}{S}\ip{S}{A^*S} - \ip{f'}{S}\\
	 &= 12\ip{f'}{S}\ip{S}{S}-\ip{f'}{S}\\
	 &=0,
 \end{align*}
where we used that $\ip{S}{S} = \int_{-1/2}^{1/2}t^2dt = 1/12$. Concerning the condition in the before last display above, it can be equivalently formulated as
\[
\ip{K_\theta(-f'-\gamma)}{K_\theta h} = \ip{-f'-\gamma}{K_{-\theta} K_\theta h} = \ip{-f'-\gamma}{(g*g)*h} = 0 \;\;\; \forall h\in H.
\]
Since $\gamma = -f'+\lambda AS$, the location preserving property of $A^*$ and of convolution with $g$ reduces the above display to
\[\ip{AS}{g*g*h} = \ip{S}{A^*g*g*h} = \ip{S}{h}= 0 \;\;\; \forall h\in H,\]
which is true  by definition of $H$. We finally compute the efficient information using \eqref{eq: info} in combination with the location preserving properties of $A^*$ and $g*\cdot$ as well as self-adjointness of $g*\cdot$.
\begin{align*}
	\tilde i^{(2)}_{\theta, f} &= \|\dot K_\theta f - K_\theta \gamma\|^2\\
	&= \|K_\theta (-f' - \gamma)\|^2\\
	&= \lambda^2\|g*AS\|^2\\
	&= \lambda^2\ip{g*AS}{g*AS}\\
	&= \lambda^2 \ip{S}{S}\\
	&=\lambda^2/12.
\end{align*}

\subsubsection*{Proof of Claim \ref{claim: location}}
For the first part of the claim, observe that by Fubini and substitution $u = t-s$, we have
\begin{align*}
	\ip{g*h}{S} &= \int_{-1/2}^{1/2} t \int_{-1/2}^{1/2} g(t-s)h(s) ds dt\\
	&= \int_{-1/2}^{1/2} h(s) \int_{-1/2}^{1/2} (u+s) g(u) du ds\\
	&= \int_{-1/2}^{1/2} h(s) \left[ s\int_{-1/2}^{1/2} g(u) du + \int_{-1/2}^{1/2} u g(u) du \right] ds\\
	&= \int_{-1/2}^{1/2}h(s)(s\cdot 1 +0)ds\\
	&= \ip{h}{S},
\end{align*}
where we used that $g$ has mass equal to 1 and that it has location equal to 0 because it is symmetric.

For the second part of the claim, let $h\in L^2[-1/2,1/2]$. We have by Fubini,
\begin{align*}
	\ip{K_\theta f}{h}_{L^2}&=\int_{-1/2}^{1/2}\int_{-1/2}^{1/2}g(t-u)f(u-\theta)\cdot h(t) dtdu\\
	&=\int_{-1/2}^{1/2}\int_{-1/2}^{1/2}g(u-t)f(u-\theta) h(t) dtdu \text{ (by symmetry of $g$)}\\
	&=\int_{-1/2}^{1/2}g*h(u) f(u-\theta)du\\
	&=\int_{-1/2-\theta}^{1/2-\theta}g*h(s+\theta) f(s)ds\\
	&=\int_{-1/2}^{1/2}g*h(s+\theta) f(s)ds \text{ (by 1-periodicity)}\\
	&=\ip{f}{K_{-\theta}h}_{L^2}.
\end{align*}

\subsection{Proof of Lemma \ref{lemma: xray_info}}

Let $(y,w) = \p_t(x,v)$. By definition of the geodesic flow, the point $y$ is at distance $t$ of $x$. The geodesic started at $y$ going backwards, i.e. along $-w$ therefore reaches the boundary point $x$ at time $t$, so that $\tau(y,-w) = t$. Besides, $F(\p_t(x,v))=(x,v)$ by definition of the footpoint map. It follows that $A_\chi (\p_t(x,v)) = \int_0^t \chi(\pi(\p_s(x,v)))dt$. The expression for $K_\theta f $ in \eqref{eq: K_theta_xray} then directly follows from \eqref{eq: attenuated_xray} applied to $a=\theta\cdot\chi$, while \eqref{eq: dev_K_theta_xray} is obtained from differentiation w.r.t $\theta$. 

The geodesic flow is smooth on $SM$, and recall that both $F$ and $\tau$ are smooth on $SM\backslash \partial_0SM$. Therefore $A_\chi\in C^\infty(SM)$ because $\chi\in C^\infty(M)$ and is supported away from $\partial M$. Similarly, if $f\in C^\infty(K)$, then $\dot K_{\theta}f\in C^\infty(\partial_+SM)$. Let now $h:=\dot K_{\theta}f$. We want to show that $K_{\theta}^*h\in C^\infty(M)$. 
The backprojection formula gives us that
\begin{align}\label{eq: last_adjoint}
	K_{\theta}^*h(y) = \int_{S_yM}h(F(y,w))e^{-\theta A_\chi(y,w)}dS(w).
\end{align}
Let now $\delta:=\text{dist}(K,M)$. We have that $\delta>0$ by assumption that $K$ is compact within the interior of $M$. Define then the open set \[U_{\delta}:=\{(x,v)\in SM: \tau(F(x,v)) <\delta\}.\]
 $U_\delta$ contains the points in $SM$ lying on ``short geodesics" that spend only a ``time"(or distance) $\delta$ within $M$ before leaving it. We will use two important facts about $U_\delta$. The first one is that by construction, those short geodesics never reach $K$, and hence if $(x,v)\in U_\delta$, then $\pi(\p_t(F(x,v)))\notin K$. The second fact is that $U_\delta\supset \partial_0SM$. This is because if $(x,v)\in \partial_0SM$, then $\tau(F(x,v))=\tau(x,v)=0$ because it lies on a glancing geodesic that never enters $M$.
 
 From the first fact above we obtain that $\forall (x,v)\in U_\delta, \ f(\pi(\p_t(F(x,v)))) = 0$. This implies that $h(F(y,w)) = 0$ if $(y,w)\in U_\delta$, and more specifically that the integrand of \eqref{eq: last_adjoint} vanishes on $U_\delta$. Since $U_\delta$ is open we obtain that all derivatives of the integrand also vanish on $U_\delta$.  In particular, the integrand is smooth on $U_\delta$.
 
 From the second fact we obtain that $SM\backslash \overline U_{\delta/2}$ is a strict subset of $SM\backslash \partial_0SM$. By smoothness of $F$ on $SM\backslash\partial_0 SM$, we thus have smoothness of the integrand on $SM\backslash \overline U_{\delta/2}$ by composition. 
 
 We thus have that the integrand is smooth on each of the open sets $U_\delta$ and $SM\backslash \overline U_{\delta/2}$ that cover $SM$, and that it vanishes on their intersection. We thus conclude that the integrand is in $C^\infty(SM)$. Taking an integral over the fiber $S_yM$ in \eqref{eq: last_adjoint} does not affect smoothness, so $K_\theta^*h\in C^\infty(M)$. 
 
 The LFD $\gamma_{\theta,f}$ is the minimiser of $\|\dot K_\theta f - K_\theta g\|_{L^2(\partial_+SM)}$ for $g\in H$. By Theorem 2.2. in \cite{monard2019}, $K_\theta^*K_\theta$ has a well defined inverse in $C^\infty(M)$. The LFD is then easily seen to be given by $(K_\theta^*K_\theta)^{-1}K_\theta^*\dot K_\theta f$. As per this theorem, we also obtain that in general, $\gamma_{\theta, f}\in \{d_M^{-1/2}g:g\in C^\infty(M)\}$ for $d_M$ any $C^\infty$ function that equals the Riemannian distance near the boundary and is positive in the interior of $M$. 
 
 
 Finally in the case that $M$ is the unit disk $\DD$, we have by Theorem 6.4 in \cite{MonardNicklPternain2021} that $K_\theta^*K_\theta$ is an isomorphism of $C^\infty(M)$ since the attenuation $i\theta\chi$ is compactly supported and trivially skew-hermitian. We thus have that $\gamma_{\theta, f}\in C^\infty(M)$ in that case.

\section*{Acknowledgments}

The authors would like to thank Richard Nickl and Gabriel Paternain for valuable remarks.

\bibliographystyle{plain}
\bibliography{biblio.bib}

\end{document}